[1994/12/01]
\documentclass[10pt,reqno]{amsart}
\usepackage[]{geometry}                
\usepackage{dsfont}
\usepackage{hyperref}
\usepackage{graphicx}
\usepackage{amssymb}
\usepackage{epstopdf}
\usepackage{amsmath}
\usepackage{amsthm,amsxtra}
\newtheorem{theorem}{Theorem}[section] 
\newtheorem{lemma}[theorem]{Lemma}
\newtheorem{coro}[theorem]{Corollary}
\newtheorem{prop}[theorem]{Proposition}
\newtheorem{remark}{Remark}[section] 
\usepackage{booktabs}
\usepackage{xcolor}

\newcommand{\sech}{\,\text{sech}\, }

\newcommand{\R}{\mathbb{R}}

\title[Inflaton and CDM decay]{A proof of slow-roll local decay of inflaton fields in cosmology and axion fields in cold dark matter models}

\author[Mat\'ias Morales]{Mat\'ias Morales}
\address{Departamento de Ingenier\'{\i}a Matem\'atica, Universidad de Chile, Casilla
170 Correo 3, Santiago, Chile.}
\email{mbmorales@dim.uchile.cl}
	\thanks{M.M. was partially funded by Chilean research grants FONDECYT 1191412 and 1231250, and Exploration ANID 13220060.}
\author[Claudio Mu\~noz]{Claudio Mu\~noz}  
\address{Departamento de Ingenier\'{\i}a Matem\'atica and Centro
de Modelamiento Matem\'atico (UMI 2807 CNRS), Universidad de Chile, Casilla
170 Correo 3, Santiago, Chile.}
\email{cmunoz@dim.uchile.cl}
\thanks{C.M. was partially funded by Chilean research grants FONDECYT 1191412, 1231250, and Basal CMM FB210005, and Exploration ANID 13220060.}

\keywords{Cosmological Inflation, Axion, Cold Dark Matter, local decay}

\numberwithin{equation}{section}

\begin{document}

\maketitle

\begin{abstract}
We consider the long time behavior of solutions to scalar field models appearing in the theory of cosmological inflation, oscillons and cold dark matter, in presence or absence of the cosmological constant. These models are not included in standard mathematical literature due to their unusual nonlinearities, which model different features with respect to classical fields. Here we prove that these models fit in the theory of dispersive decay by computing new virials adapted to their setting. Several important examples, candidates to model both effects are studied in detail.   
\end{abstract}


\section{Introduction}


This paper is concerned with the long time behavior of solutions to some cosmological inflationary and Cold Dark Matter (CDM) models.

\subsection{The theory of Inflation}  Hubble \cite{Hu}, based on the observation of local galaxies, observed that the universe expands and proposed the famous formula $v=H_0 d$, where $v$ is the radial velocity of nearby galaxies and $d$ its distance. The parameter $H_0$, today known as the Hubble's constant (or parameter) describes the expansion of our universe. A posterior analysis, with the discovery of the Cosmic Microwave Background (CMB) \cite{PW}, a vestige of the decoupling between matter and radiation, confirmed the idea of a primordial universe more compact than the actual one. A classical assumption is that our large-scale universe is governed by the cosmological principle, with a Universe homogeneous and isotropic. In this setting, the Friedmann-Lemaitre-Robertson-Walker (FLRW) metric
\begin{equation}\label{FLRW}
ds^2 = -dt^2 +a^2(t) \left( \frac{dr^2}{1-Kr^2} +r^2 d\Omega^2\right)
\end{equation}
 is commonly used to describe the Universe at great scales and it is one of the basis of the current Big-Bang theory, usually referred as the $\Lambda$CDM cosmological model. The other key component of the model is the CDM theory. 

\medskip

A key component of the FLRW metric  \eqref{FLRW} is the \emph{scale factor}, usually denoted as $a(t)$, which measures either the expansion or contraction of the Universe with respect to a time scale variable. The parameter $K$ measures the curvature of the universe, being $K=0$ a flat spacetime. Assuming this as a perfect fluid leads to the Friedmann's equations, one has an energy momentum tensor
\[
T_{\mu\nu} = (\rho +p) u_{\mu}u_{\nu} +p g_{\mu \nu}, \quad p= \hbox{pressure}, \, \rho=\hbox{density}, 
\]
that through Einstein's field equations couples the scalar factor with the content of the universe, in the sense that 
\[
\dot \rho + 3H (\rho + p)=0, \quad H(t):= \frac{\dot a(t)}{a(t)}.
\]
Usually, and also in this paper, we shall assume that $H$ is constant. The precise value of today's Hubble parameter (and the consequent equipartition of mass-energy of the universe) is matter of a hot controversy \cite{H01,H02}, between essentially two methods of measuring  $H$ (among other important observables) that differ in their outputs: measures using the local distance ladder and those  inferred using the CMB and galaxy surveys.  Density $\rho$ in the Universe is a tough question. Precisely, this last ingredient is another issue in Cosmology, see Subsection \ref{ss1p2}.

\medskip

However, the Big-Bang theory as it is does not explain several puzzling observations of our current universe. These are the \emph{horizon, the flatness, and the initial conditions} problems. The horizon problem refers to the impressive homogeneity of our universe, given the lack of causal connection among extreme sections of it. The flatness problem corresponds to the extreme flatness of the Universe today, given the fact that expansion and time evolution should make our universe even flatter. Finally, the Big-Bang model does not explain the initial conditions required to fulfill today's universe. Precisely, the theory of inflation was introduced to repair these unsolved issues of the model, and solves with great success the two first problems. For the third one must consider perturbations of inflation.  

\medskip

Inflation is essentially a phenomenological theory and it is modeled via a classical scalar field coupled to gravitation, usually called \emph{inflaton}. Since the period of inflation is the primordial Universe, perturbations of the inflaton potential must also consider quantum effects, random perturbations of gaussian type that recover with precision the observed data. The theory of cosmological inflation was first introduced by Guth \cite{Guth} and Starovinski \cite{Starovinski}. The inflation field is characterized by a nonstandard well potential, and it is suggested that it may be not related to the classical $\phi^4$ model present in the Standard Model of particle Physics. Instead of tunneling out of a false vacuum state, inflation occurred by the scalar field rolling down a potential energy hill. When the field rolls very slowly compared to the expansion of the Universe, inflation occurs. However, when the hill becomes steeper, inflation ends and reheating can occur.

\medskip

Since it is still an unproved theory, many potentials have been proposed to describe this theory and its perturbations, some of them with better chances due to current experimental observations of the CMB. Among these experimental projects, Planck \cite{Planck2018} has worked analyzing statistically anisotropies of the CMB to understand, among other questions, which of these models could give a better description of our universe.

\subsection{A theory for Cold Dark Matter}\label{ss1p2} The composition of the Universe is another puzzling problem in Cosmology, and it is deeply related to the previous discussion of inflation. Today is widely recognized that most of our universe is composed of radiation, baryonic matter, dark matter and dark energy. The last two are essentially not well-understood: dark matter was proposed to explain anomalous rotation speeds in galaxies, and dark energy seeks to explain the acceleration of the universe. Dark matter is detected only through its gravitational interactions with ordinary matter and radiation. 

\medskip

There are several theories that proposes to explain dark matter. The primary candidate for dark matter is some new kind of elementary particle that has not yet been discovered, particularly weakly interacting massive particles (WIMPs).  One can also find Axions \cite{SW} and Primordial Black Holes (PBH) \cite{Kawa}. The latter are hypothesized compact objects of the early universe formed by strong curvature deviations and not by accumulation of mass as astrophysical black holes. They have been under review because recent detections of gravitational waves by LIGO interferometers showed unusual ranges of masses. It is also believed that PBH might have important implications in quantum perturbations of the inflaton field \cite{Palma}. In this paper we will concentrate our efforts in working with Axion fields as the ones described in \cite{BrZh,Kawa,Hong}, and in particular, in the theory of \emph{oscillons}. Axions are proposed elementary particles that under low mass constraint, are of interest as a possible component of cold dark matter. Indeed, in recent years axions have become one of the most promising candidates for dark matter \cite{AX}.

\subsection{Inflation's mathematical description} Mathematically speaking, the theory goes as follows. The setting is the one given by Einstein's field equations for a Lorentzian metric $g$ in 1+3 dimensions. Let $\phi: \R^{1+3} \to \R$ be the perturbation of the inflaton field. His action is given by
\begin{equation}\label{action}
\mathcal{S} = \int_{\R^{1+3}}\sqrt{-g}\left(\dfrac{1}{2}g^{\mu\nu}\partial_\mu \phi \partial_\nu \phi - F(\phi)\right),
\end{equation}
where $g=g_{\mu \nu}$ the metric of the spacetime.
Here we always assume that the universe is spatially flat, homogeneous and isotropic. Specifically, we consider a de Sitter universe, that is, the metric $(g_{\mu\nu})$ takes the form 
\[
ds^2 = dt^2-e^{2Ht}(dx^2+dy^2+dz^2),
\]
where $H \geq 0$ represents the Hubble parameter, which we will assume constant. The Euler Lagrange equation associated to \eqref{action} is
\begin{eqnarray}\label{eq:1}
\partial_t^2\phi+3H\partial_t\phi- e^{-2Ht}\Delta \phi+f(\phi) = 0,
\end{eqnarray}
%
%
%
with the notation $f = F'$. The Laplacian $\Delta=\Delta_x$ here is taken in the spatial variable $x$. The nonlinearity $f=f(s)$ is unusual, in the sense that it will satisfy particular conditions natural for the study of inflation \cite{Amin} and or CDM \cite[eqn. (15)]{Lyth}. This general setting is particularly useful since one can also consider another important unsolved problem, the dark matter existence, for which one of the most promising theories looks quite similar, and the corresponding field is represented by another particle called axion.

\medskip

The equation \eqref{eq:1} has no conserved energy in general, except when $H = 0$. Instead, it formally satisfies the following relation 
\begin{equation}\label{Energy}
\begin{aligned}
E[\phi,\phi_t]:= &~{}\int_{\R^3} \left( \dfrac{\phi_t^2}{2}+\dfrac{|\nabla \phi|^2}{2e^{2Ht}}+F(\phi) \right) , \\
\dfrac{d}{dt}E[\phi,\phi_t]= &~{} -H\int_{\R^3} \left(3\phi_t^2+\dfrac{|\nabla \phi|^2}{e^{2Ht}} \right).
\end{aligned}
\end{equation}
This is a remarkable difference between the stationary ($H=0$) and non stationary ($H>0$) universe scenarios. It is widely accepted in Physics that today's universe satisfies $H>0$. Finding the actual value of $H$ is an active research topic, because there exists an inconsistency between the values inferred by different experiments, ones using local measurements of the current expansion and others using the $\Lambda$CDM model together with the CMB data  Notice that when $H = 0$ equation \eqref{eq:1} becomes the usual non linear wave equation and we recover the classical conservation of energy. Previous results in the case $H>0$ and focusing nonlinearities of F. John type reveal that blow up exists under a sign condition on the initial data \cite{TW}.	

\medskip

The Klein-Gordon model \eqref{eq:1} has been worked many times in the literature, specially in the case with constant coefficients and zero Hubble parameter. In this case the literature is extensive. The reader may consult the review \cite{KMM3} and references therein for a detailed account of results.

\subsection*{Acknowledgments} We deeply thank Gonzalo Palma for several illuminating discussions and comments that strongly helped us to better understand inflation theory.   

\subsection*{Organization of this paper}

This paper is organized as follows: In Section \ref{models} we discuss briefly the potentials to be studied and the physical context where they appear, and in Section \ref{results} we enunciate the general main results that will be proved on this work. In Section \ref{preliminaries} some results about Sobolev spaces are recalled, as well as classical existence and uniqueness theorems for nonlinear wave equations. Section \ref{viriales} is devoted to prove virial identities that later will be used in Section \ref{proofs} to prove Theorems \ref{thm: decaimiento}, \ref{thm: decaimiento 2} and \ref{thm: decaimiento 3}. Finally, in Section \ref{Applications} we use the theorems proved to study in detail the dynamics of the models discussed in \ref{models}.

\section{Slow-roll Inflationary and Cold Dark Matter models}\label{models} 

In this section we will present the inflationary and CDM models to be studied in this work. Some of them have been already studied in other works, but many have never been considered in the mathematical literature.

\subsection{$E$ and $T$ models} From Planck \cite[Table 5]{Planck2018} one can access to a selection of slow-roll inflationary models of high interest in order to explain cosmological inflation. Among the most favorable models we highlight the Starobinsky $R^2$, $f(R)$ modified gravity or $E_1$ model, represented by the potential
\begin{equation}\label{F11}
F_{1,1}(\phi) := (1-e^{-\phi})^2.
\end{equation}
Notice that $F_{1,1}$ is a potential exponentially unbounded as $\phi\to -\infty$ (see Fig. \ref{fig:F1n}), with some very unpleasant features. Among them, we can find
\[
F_{1,1}(0) =F_{1,1}'(0)=0, \quad \hbox{but} \quad F_{1,1}''(0)=2>0.
\]
This last positive sign makes mathematical treatment of the small data theory not easy. Cosmological theory supposes that the initial configuration starts with $\phi \gg 1$, and slowly decays in time towards the zero field value. This process is called the ``slow-roll'' dynamics, and the exponential growth of the scaling parameter of the universe ($a(t) \sim e^{2Ht }$) is described as the ``e-fold'' procedure. 

\begin{figure}[h]
   \centering
   \includegraphics[scale=0.35]{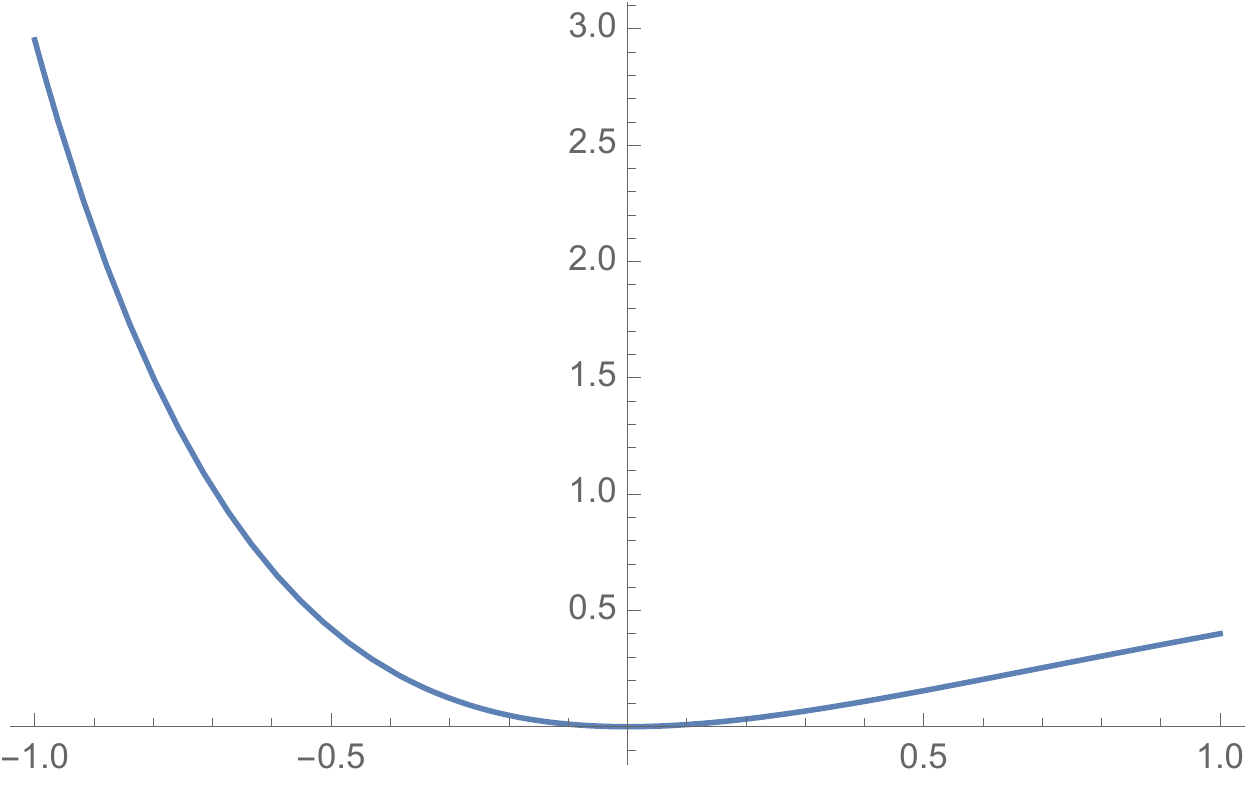}
   \includegraphics[scale=0.35]{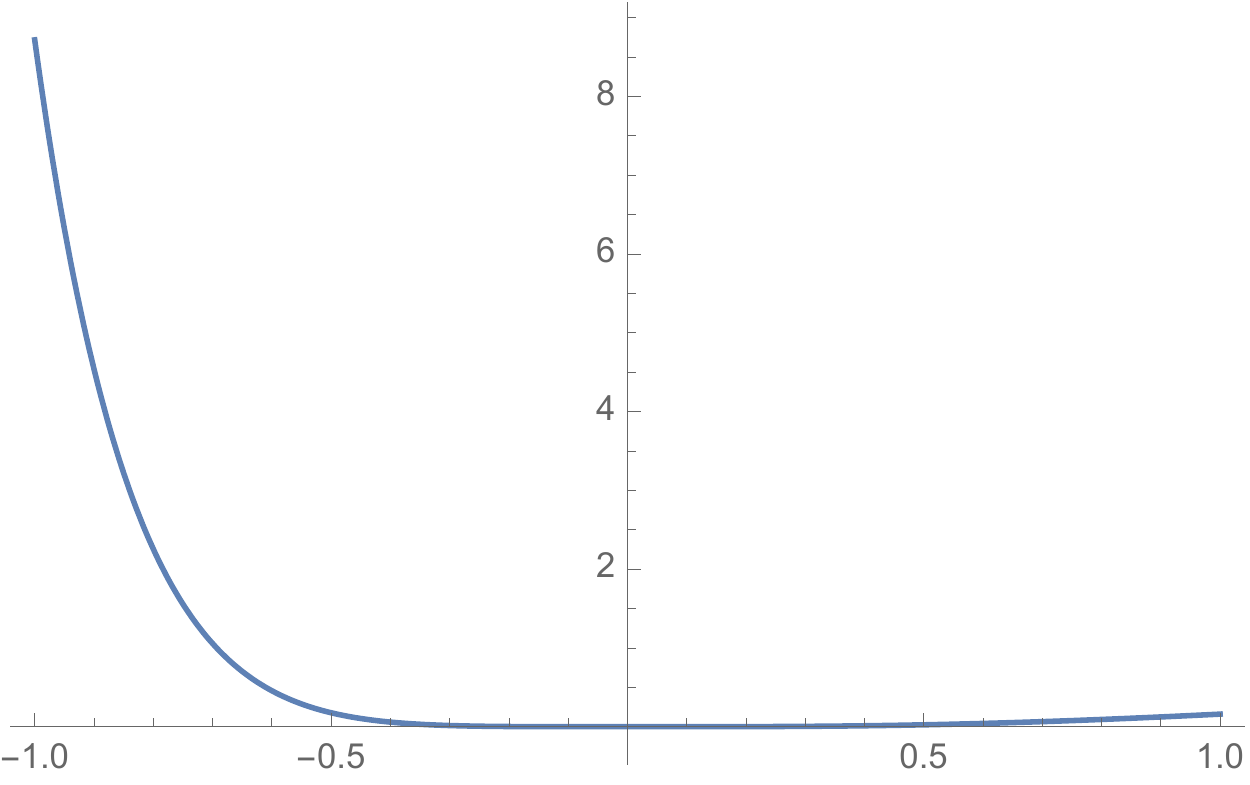}
   \includegraphics[scale=0.35]{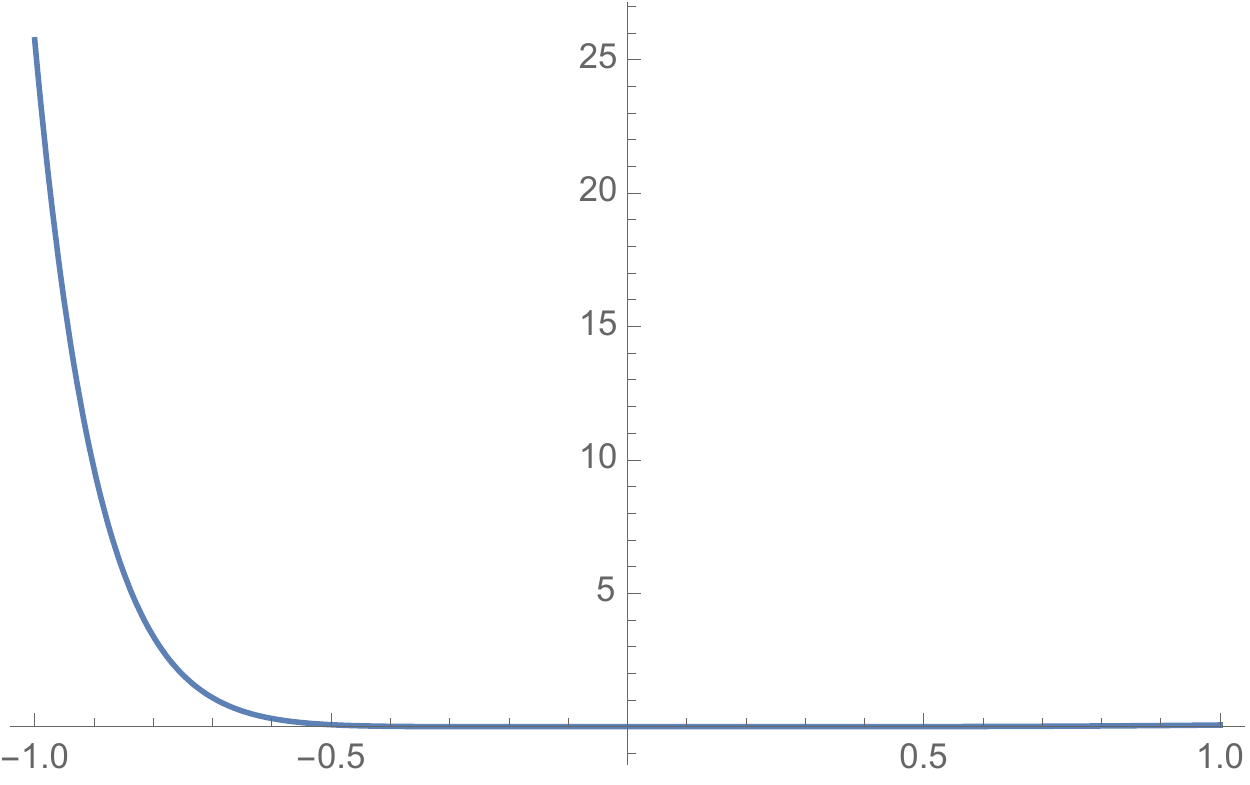} 
   \caption{The potentials $F_{1,1}$ \eqref{F11}, $F_{1,2}$ and $F_{1,3}$ in \eqref{F1n}.}
   \label{fig:F1n}
\end{figure}

The $F_{1,1}$ model is part of a family of inflationary potentials that gives rise to the so-called $E_n$ theories:
\begin{equation}\label{F1n}
F_{1,n}(\phi) := (1-e^{-\phi})^{2n}, \quad n\geq 1,
\end{equation}
see Fig. \ref{fig:F1n}. For us the most interesting cases are the ones with $n=1,2$. Additionally to the $E_n$ models, one has the $T_n$ ones, which are also highly relevant in Planck data analysis. These are given by
\begin{equation}\label{F2n}
F_{2,n}(\phi) := \tanh^{2n}(\phi), \quad n\geq 1,
\end{equation}
see Fig. \ref{fig:F2n}. The case $n=1$ is highly favorable in our setting, producing the best result of this paper, but $n=2$ has some drawbacks due to the lack of a sign condition. Only small data will be suitable to prove decay.
\begin{figure}[h]
   \centering
   \includegraphics[scale=0.35]{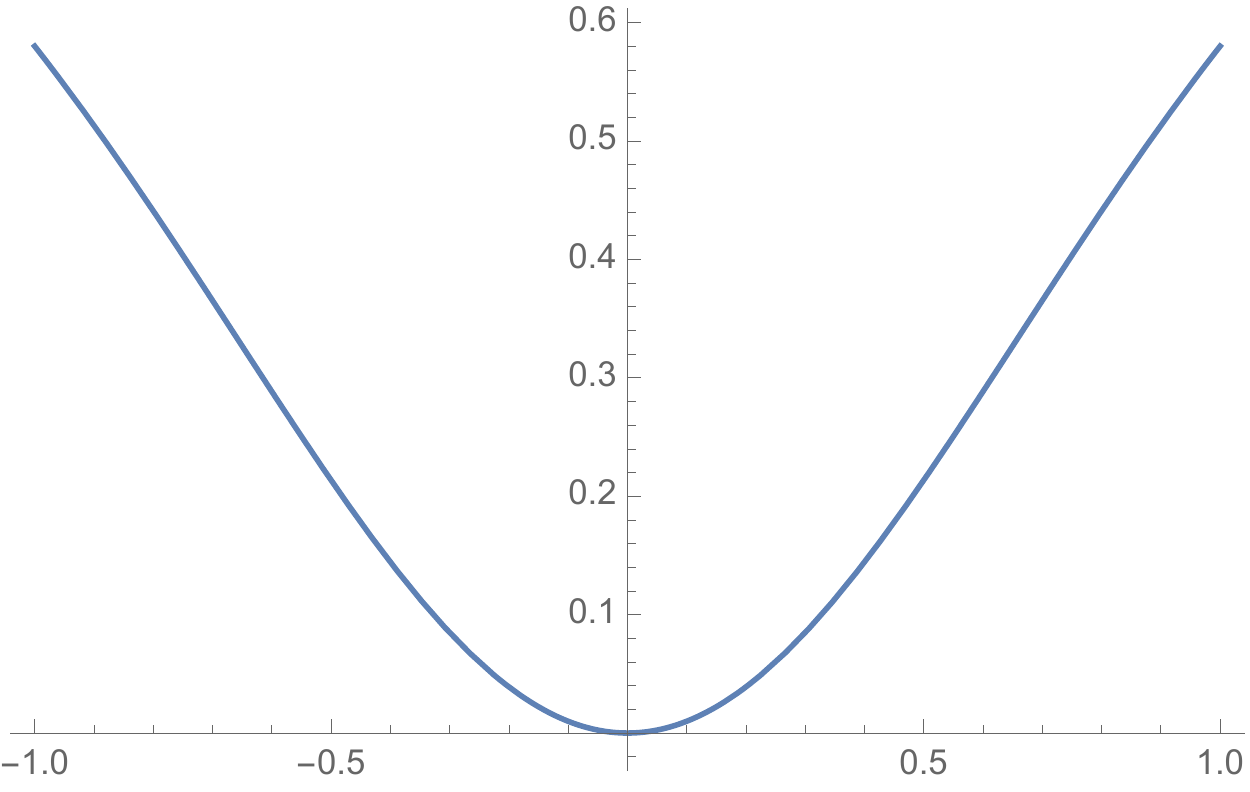}
   \includegraphics[scale=0.35]{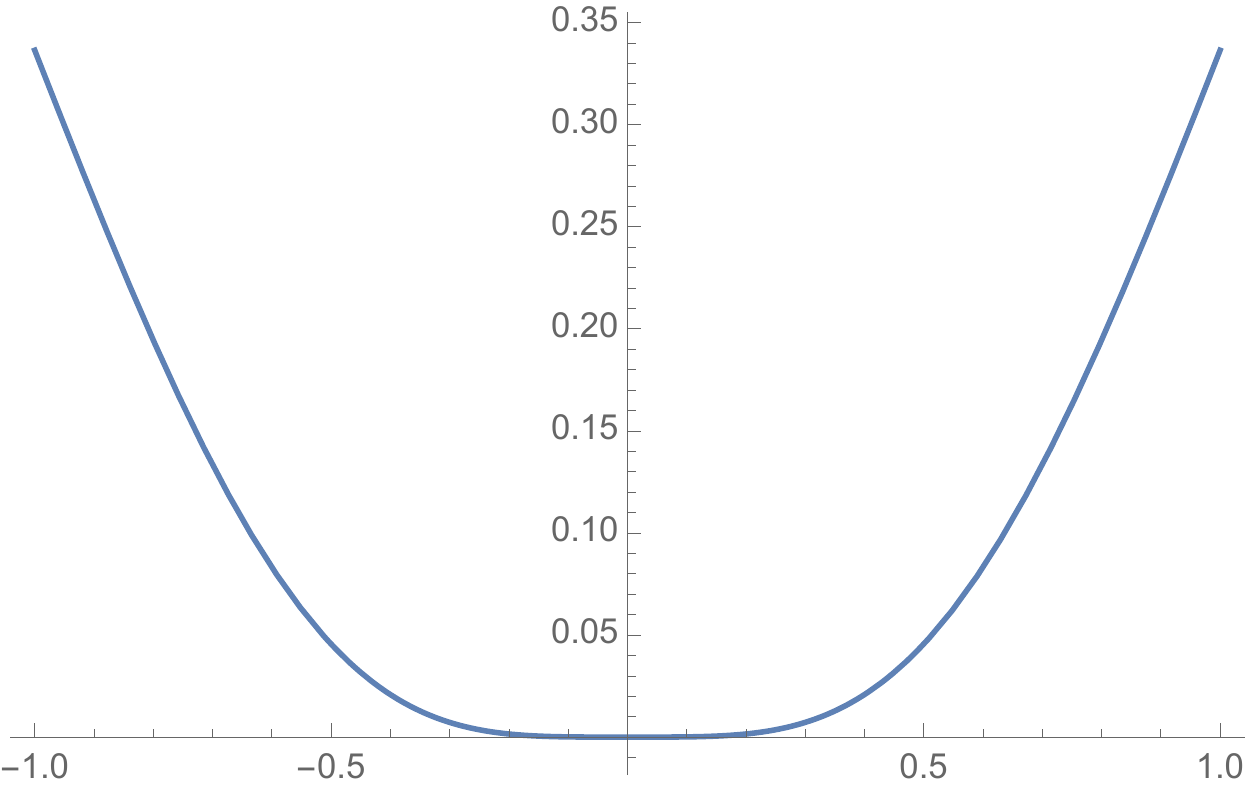}
   \includegraphics[scale=0.35]{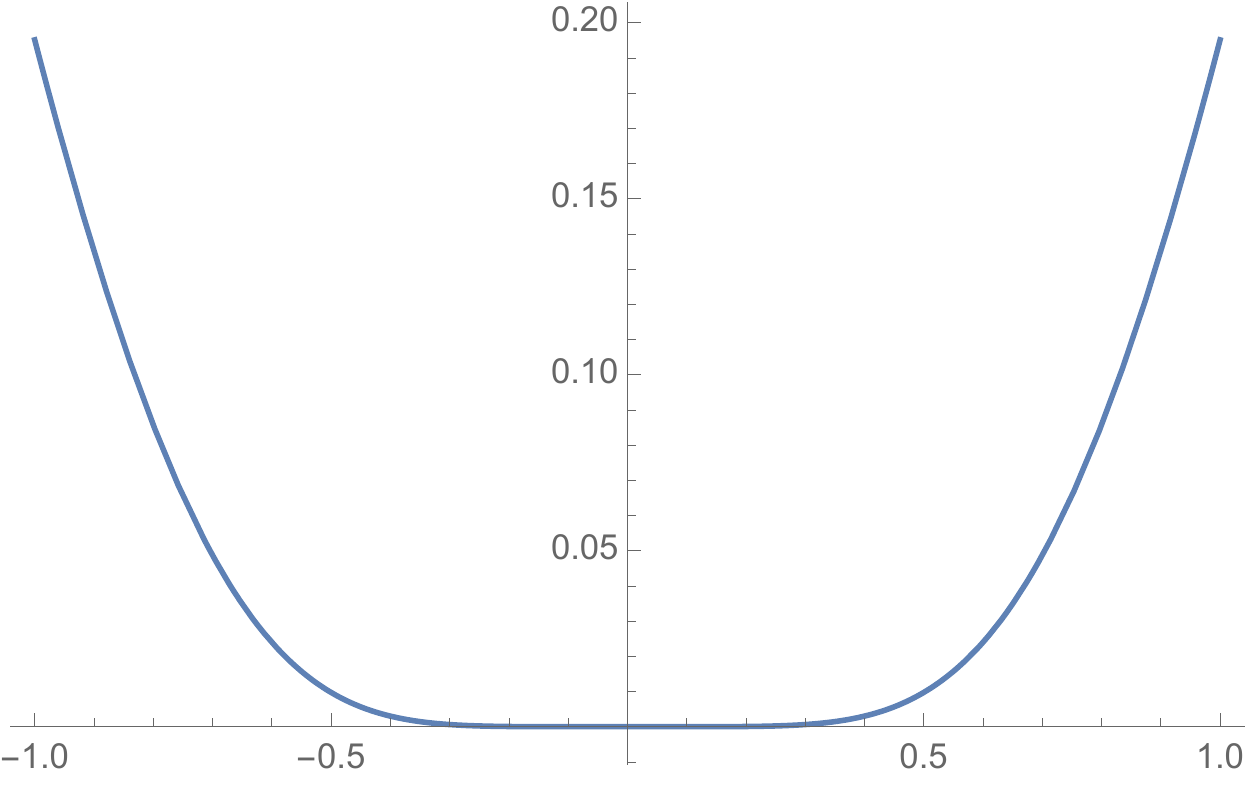} 
   \caption{The potentials $F_{2,n}$ in \eqref{F2n}.}
   \label{fig:F2n}
\end{figure}

\subsection{Natural Inflation and Axion potentials}

There are other models of equal interest and high importance in the quest for the inflaton potential. These are the so-called Natural inflation (- sign) \cite{Planck2018} and Axion potential (+ sign) \cite[p. 4]{BrZh}  (see Fig. \ref{fig:F3})
\begin{equation}\label{F3}
F_{3,\pm}(\phi):= 1 \pm \cos \phi.
\end{equation}
In both cases it is assumed that the field $\phi$ is no larger than $\frac{\pi}2$ in absolute value. The potential $F_{3,\pm}$ is the classical appearing in 1D sine-Gordon models, making the scalar field model integrable. In 3D the situation is different, and in radial symmetry integrability seem lost. In both cases we are able to give answers to the decay problem. 
\begin{figure}[h]
   \centering
    \includegraphics[scale=0.35]{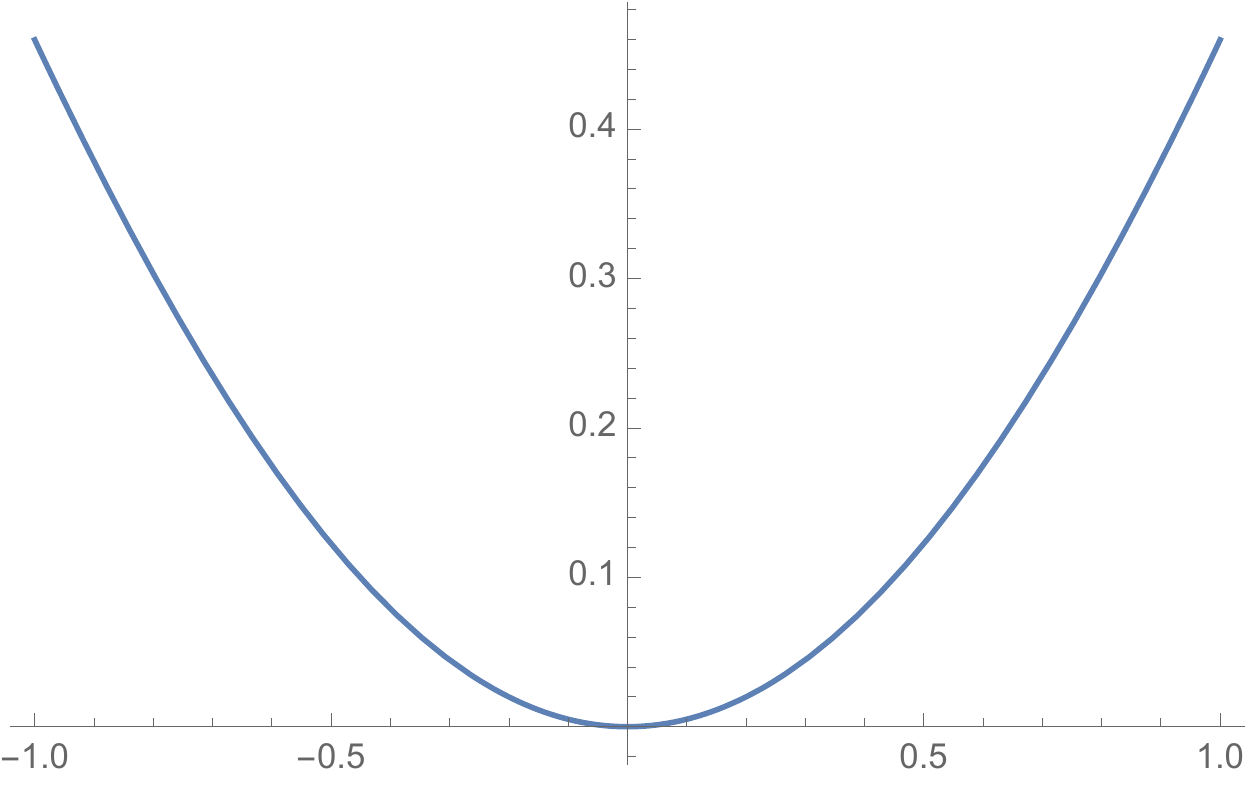}\quad
       \includegraphics[scale=0.35]{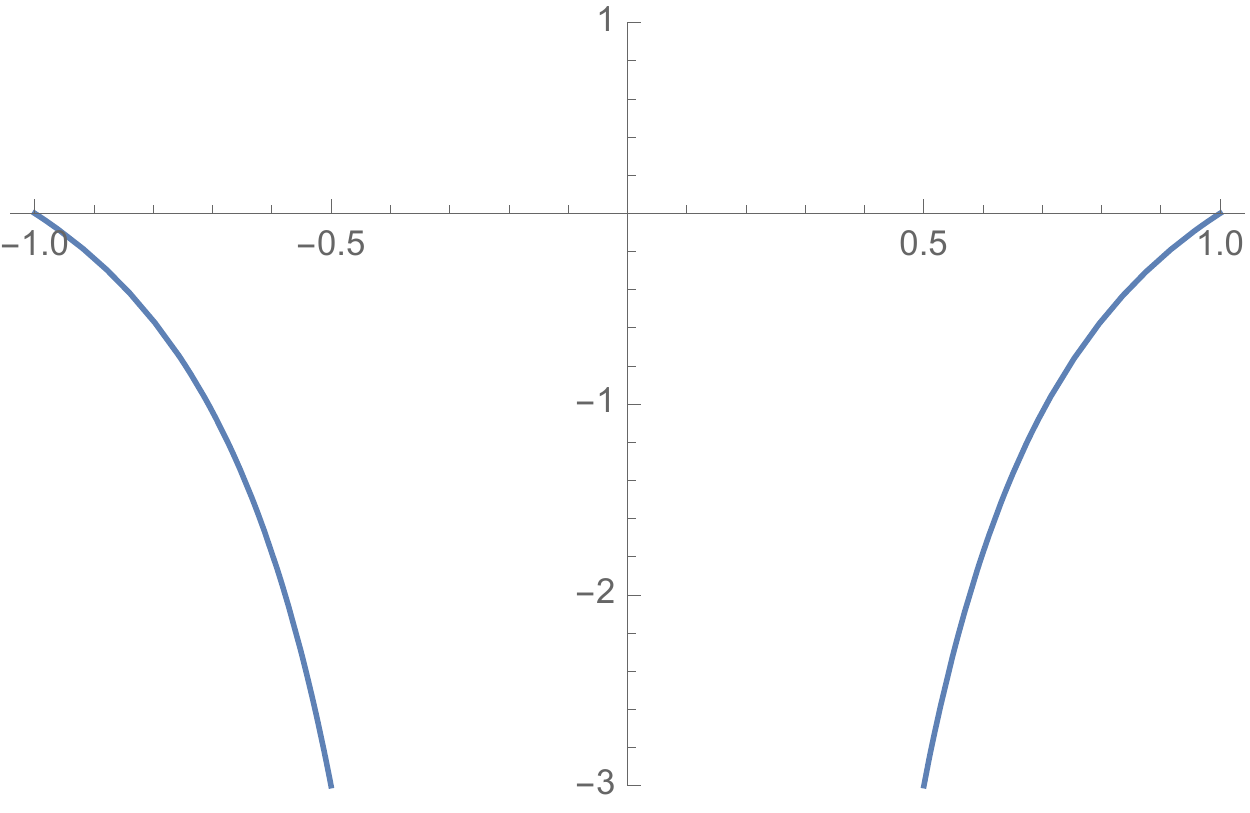}
   \caption{Left: The potential $F_{3,-}$ from \eqref{F3}. Right: The potentials $F_{4,n}$, $n=1$ in \eqref{F4}. Note that the last potential is singular at $\phi=0$.}
   \label{fig:F3}
\end{figure}
%

\subsection{The $D$-brane and Hilltop models}
In addition to the sine-Gordon models, another important model is given by the singular $D$-brane (Fig. \ref{fig:F3}):
\begin{equation}\label{F4}
F_{4,n}(\phi) := 1- \frac{1}{\phi^{2n}}, \quad n=1,2, \quad \phi\neq 0.
\end{equation}
Notice that the physical problem here is the perturbation of a $\phi$ large initial state. In this case we shall assume
\[
\phi=1 + v, \quad |v|\ll 1, 
\]
so that after renormalization (to have finite energy) we will work with the modified potential
\begin{equation}\label{tF4}
\tilde F_{4,n}(v):= 1- \frac{1}{(1+v)^{2n}} - 2nv = \frac{(1+v)^{2n} -1-2nv(1+v)^{2n}}{(1+v)^{2n}}.
\end{equation}
Also considered in this paper will be the Hilltop models (Fig. \ref{fig:F5n}):
\begin{equation}\label{F5n}
F_{5,n}(\phi) := - \phi^{2n}, \quad n=1,2.
\end{equation}
The case $n=1$ is exactly linear Klein-Gordon and will not be studied in this paper. However, the case $n=2$ is highly interesting because it behaves as one of the most promising potentials to describe inflation.

\subsection{Axion-Monodromy and log potentials}
The last two examples that we will study here also appear when studying CDM. These are the axion monodromy potential \cite{Zhang}
\begin{equation}\label{F6q}
F_{6,q}(\phi):= \frac1q \left((1+ \phi^2)^{q/2}-1 \right), \quad q\in [-1,1], \quad q\neq 0,
\end{equation}
and the logarithm potential (Fig. \ref{fig:F5n}):
\begin{equation}\label{F7}
F_7(\phi):=\frac12 \log (1+ \phi^2).
\end{equation}
The potential $F_{6,q}$ formally converges to $F_7$ as $q\to 0$. 
\begin{figure}[h]
   \centering
     \includegraphics[scale=0.35]{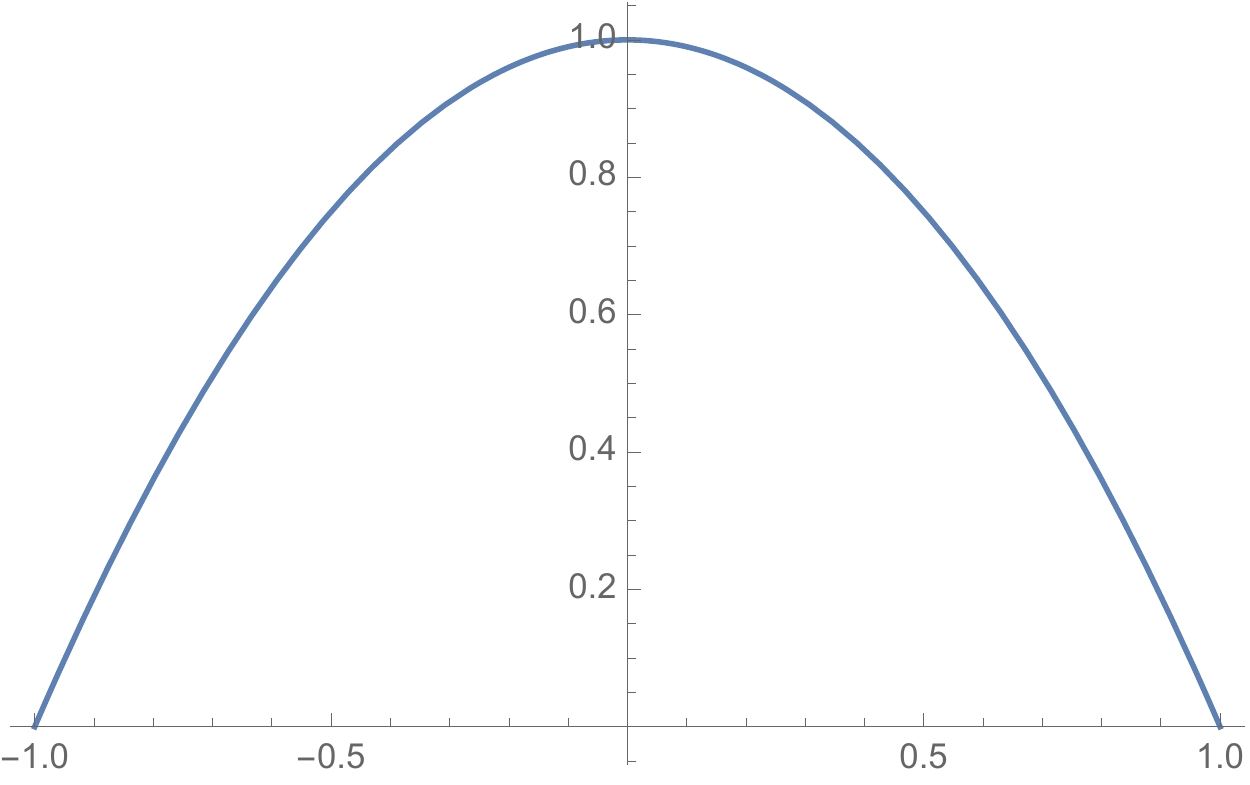} 
       \includegraphics[scale=0.35]{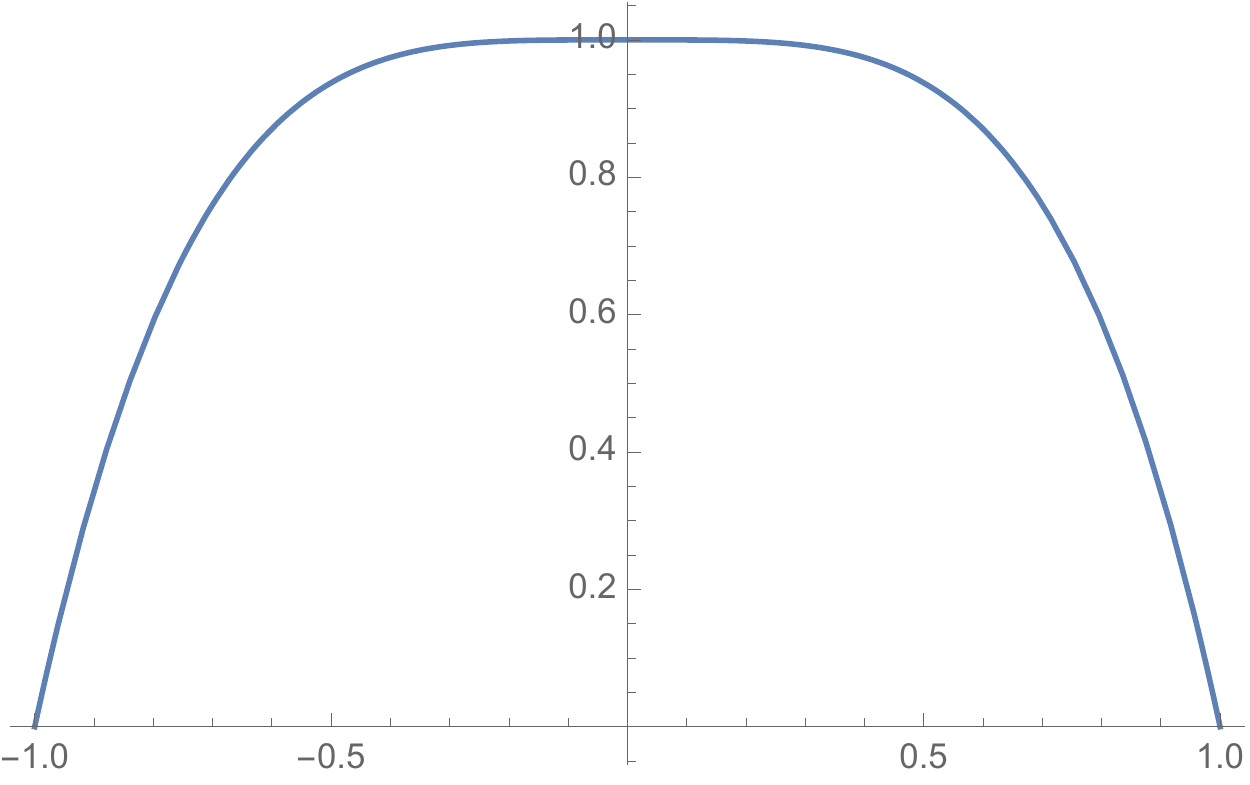} 
         \includegraphics[scale=0.35]{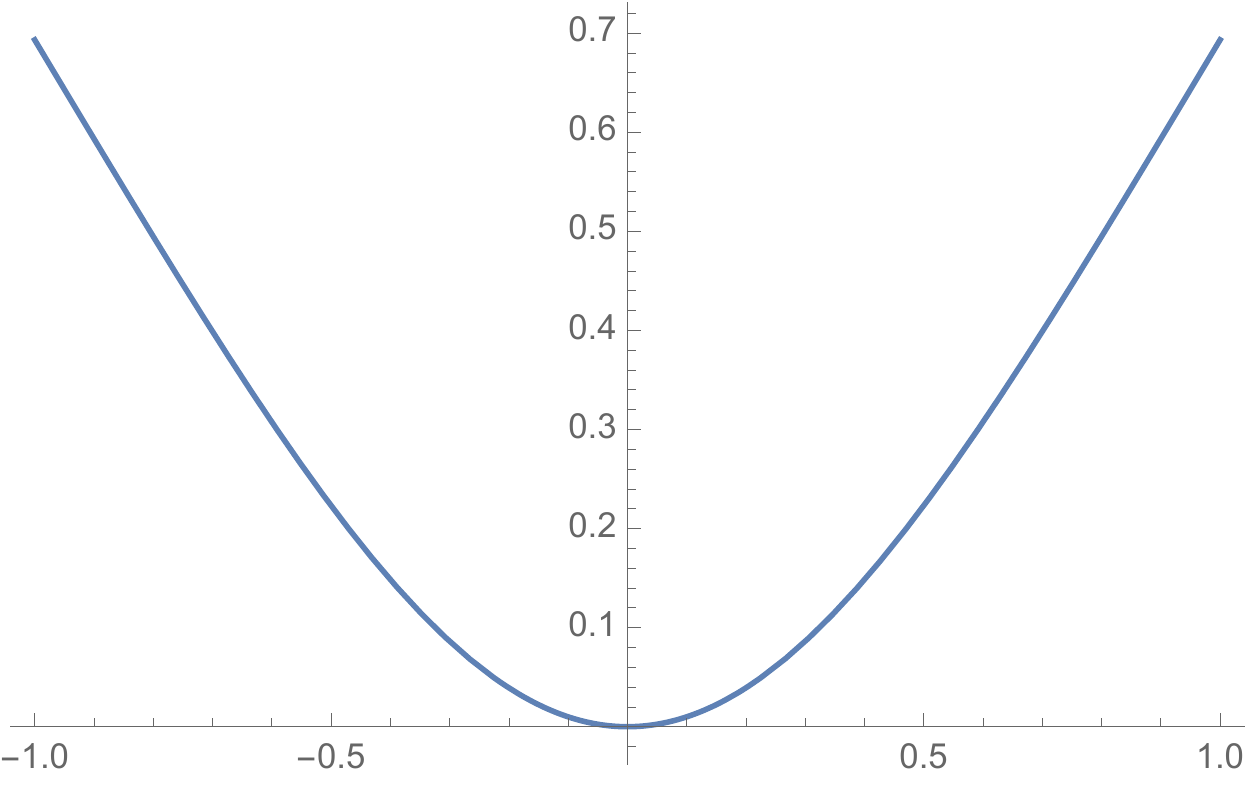} 
   \caption{The potentials $F_{5,n}$ , $n=1,2$ in \eqref{F5n} and $F_7$ from \eqref{F7}.}
   \label{fig:F5n}
\end{figure}

\medskip

These families have their own properties, usually not being part of the standard local and global well-posedness theory appearing in the literature. For each of these models, we will prove local and/or global well-posedness, and provide a proof of decay under suitable assumptions on the initial data.


\medskip

However, as far as we know, the long time behavior of the model \eqref{eq:1} with nonlinearities such as the ones considered here (and in practice, in cosmology and dark matter) have not been considered before. In that sense, our results show that strong decay is a common denominator in cosmologically motivated scalar field models.

\section{Main results}\label{results}

We shall assume that $f: \R \to \R$ is a piecewise $C^1$ function such that $f(0) = 0$, and the initial data $(\phi ,\partial_t \phi )\big|_{\{t=0\}}=(\phi _0,\phi _1) \in H^1\times L^2$ in \eqref{eq:1} is radially symmetric; consequently, from the local well-posedness theory one has that the solution $(\phi ,\partial_t \phi)$ is radial for all time whenever it is well-defined.

\medskip

Our first result is a sufficient condition to obtain local decay in the energy space. 

\begin{theorem}[Large data case] \label{thm: decaimiento} Let $f$ be globally Lipschitz and satisfying the sign conditions $F(\phi)\geq 0$ and $2F(\phi)-\phi f(\phi) \geq 0$. Then the solution $(\phi , \partial_t \phi )$ of (\ref{eq:1}) with $H = 0$ is defined for all times and 
\begin{equation}\label{decaimiento}
\lim_{t \to \infty} \|(\phi,\partial_t \phi)(t)\|_{H^1 \times L^2(B(0,R))} = 0,
\end{equation}
for any $R >0$.
\end{theorem}
Note that Theorem \ref{thm: decaimiento} is valid for any size of data and no presence of the cosmological constant. 

\begin{coro}\label{coro1}
The $E$ model $F_{2,1}$ in \eqref{F2n}, the Axion-monodromy model $F_{6,q}$ in \eqref{F6q}, and the logarithm mode $F_7$ in \eqref{F7} satisfy \eqref{decaimiento}. 
\end{coro}
Recall that the first potential appears as one of the most interesting candidates to represent inflation according to Table 5 of \cite{Planck2018}. See Section \ref{Applications} for further details.

\medskip

A second corollary from Theorem \ref{thm: decaimiento} is the following:

\begin{coro}\label{coro3}
Under \eqref{decaimiento}, the corresponding scalar field model has no standing waves nor breathers.  
\end{coro}

By breathers, we mean localized in space, time periodic solutions. See also \cite{KMM1} for similar results as Corollary \ref{coro3} in the case of small odd data and odd nonlinearities.

\medskip

Our second result concerns the case where the condition $2F(\phi)-\phi f(\phi) \geq 0$ is not satisfied, still in the case $H=0$. Several cosmological models are in this class. Our result is now local decay under smallness assumption and growth below a critical power.
  
\begin{theorem} \label{thm: decaimiento 2}
If $f: \R \to \R$ is of class $C^1$ and satisfies that for some $C, \delta >0$ 
\begin{eqnarray} \label{crecimiento}
0\leq \phi f(\phi) \leq C \phi^4, \quad \forall \phi \in (-\delta, \delta),
\end{eqnarray}
or
\begin{eqnarray} \label{crecimiento 2}
2F(\phi) - \phi f(\phi) \geq 0, \quad \forall \phi \in (-\delta, \delta),
\end{eqnarray}
then any global solution $\phi \in C([0,\infty); H^1\times L^2)$ of (\ref{eq:1}) with $H = 0$ such that 
\begin{equation}\label{acota}
\sup_{t \geq 0} \|\phi(t)\|_{H^1 \cap L^\infty} \leq  \varepsilon,
\end{equation}
satisfies 
\begin{equation}\label{convergence}
\lim_{t \to \infty} \|(\phi,\phi_t)(t)\|_{H^1 \times L^2(B(0,R))} = 0,
\end{equation}
for any $R>0$ provided that $\varepsilon >0$ is small enough.
\end{theorem}

\begin{remark}
Note that the condition $\phi f(\phi) \geq 0$ ensures that $F(\phi)\geq 0$, and the model has defocusing character. Condition \eqref{crecimiento} describes that the model has sufficient flatness at the origin to allow decay, and finally condition \eqref{crecimiento 2} it is just an application of Theorem \ref{thm: decaimiento} to the case of small data and needs no proof.
\end{remark}

Among the potentials that satisfy \eqref{convergence} (under solutions satisfying \eqref{acota}) we find the following.

\begin{coro}\label{coro2}
The $E$ model $F_{1,n}$ in \eqref{F2n} and the $T$ model $F_{2,n}$ in \eqref{F2n}, for $n\geq 2$, satisfy \eqref{convergence}. The Natural Inflation model $F_{3,-}$ in \eqref{F3} satisfies \eqref{convergence}.  Finally, the Hilltop model $F_{5,2}$ also satisfies \eqref{convergence}.
\end{coro}

Notice that one of the most important models, the so-called $R^2$-model
\[
F_{2,1}(s) = (1-e^{-s})^{2}, \quad s\in\R,
\]
does not fit in the assumptions of this result. Additionally, the case of the D-brane model \eqref{F4}, if $\phi$ is chosen of the form $\phi= 1+ v$, one should expect decay, however that case has remain elusive to us. Finally, the Axion model $F_{3,+}$ in \eqref{F3} has proved to fall outside the scope of Theorems \ref{thm: decaimiento} and \ref{thm: decaimiento 2}.
\medskip

\noindent
{\bf The case with cosmological constant.} Now we turn into the case of our current universe, assuming $H>0$. Here we have the following results:

\begin{theorem}\label{thm: decaimiento 3}
Consider equation \eqref{eq:1} with initial data $(u_0,u_1) = (\varepsilon g,\varepsilon h)$, $g,h \in C_0^\infty(\R^3)$. If $H >0$, $F(x) \geq 0$ for all $x \in \R$ and there exist $\delta,M >0$ such that \[
|f'(\phi)| \leq M\phi^2, \quad \forall \phi \in (-\delta,\delta),
\]
then there exist $\varepsilon >0$ such that the solution to \eqref{eq:1} is global in time. Moreover,
\begin{enumerate}
\item The energy of any global solution  decays to zero outside of the forward light cone, that is 
\begin{equation}\label{H00}
\lim_{t \to \infty} \int_{R(t)} \left( \dfrac{\phi_t^2}{2}+\dfrac{|\nabla \phi|^2}{2e^{2Ht}}+F(\phi) \right)= 0, 
\end{equation}
where $R(t) = \{x \in \R^3 \; | \; |x| > (1+b)t\}$ for any $b >1$.
\medskip
\item If $R>0$ is fixed,
\begin{equation}\label{H01}
\lim_{t \to \infty} \int_{B(0,R)} \left( \dfrac{\phi_t^2}{2}+\dfrac{|\nabla \phi|^2}{2e^{2Ht}}+F(\phi) \right)= 0.
\end{equation}
\end{enumerate}
\end{theorem}
Estimate \eqref{H01} shows that locally the energy must converge to zero, however, the global energy, although decreasing, may not converge to zero in general. Their main outcome should depend on the existence of moving solitary waves. In the case of radial data, this is strongly unlikely, but solitary rings of finite energy might exists. 

\begin{remark}
The global existence ensured in Theorem \ref{thm: decaimiento 3} can be weakened to the simple assumption $(\phi,\phi_t)\in H^1\times L^2$ radial and bounded in time of finite energy; however this assumption seems not simple to prove in the case of general nonlinearities and we have preferred to give a more restrictive but evidently nonempty set of data for which Theorem \ref{thm: decaimiento 3} holds.
\end{remark}

\section{Preliminaries}\label{preliminaries}

In this section we gather some results needed for the proof of our main result. We first deal with Sobolev estimates in the case of radially defined functions.

\subsection{A density lemma}
The following result is standard in the literature, but we include its proof for completeness reasons.

\begin{lemma} \label{lemma: density}
If $n \geq 3$ then $C_0^\infty(\R^n\setminus \{0\})$ is dense in $H^1(\R^n)$. Moreover, if $u \in H^1(\R^n)$ is radial and $u_k \in C_0^\infty(\R^n\setminus \{0\})$ is such that $u_k \to u$ in $H^1(\R^n)$, then we can choose $(u_k)_k$ also radial.
\end{lemma}

\begin{remark}
Notice that Lemma \ref{lemma: density} is not valid in dimensions 1, because by Sobolev's embedding one has $H^1(\R) \hookrightarrow C(\R)$. Since $H^1$ convergence implies uniform convergence, if the Lemma holds we could conclude that $u(0) = 0$ for every $u \in H^1(\R)$.  An analogous argument works in the case $H^2(\R^3)$. This Lemma is still true in dimension 2, but the proof is more complicated and is not relevant to this work.
\end{remark}

\begin{proof}[Proof of Lemma \ref{lemma: density}]
Since $C_0^\infty(\R^n)$ is dense in $H^1(\R^n)$ is enough to prove that 
\[
C_0^\infty(\R^n) \subseteq \overline{C_0^\infty(\R^n\setminus \{0\})}^{\| \|_{H^1}}.
\]
Let $u \in C_0^\infty(\R^n)$ and $\theta \in C^\infty(\R^n)$ such that 
\begin{equation*}
\theta(x) = 
\begin{cases}
0 & \text{ if }|x| <1 \\
1 & \text{ if }|x| >2.
\end{cases}
\end{equation*}
Since $\theta(kx) \to 1$ a.e. we have that $u_k(x) = u(x)\theta(kx) \to u(x)$ in $L^2(\R^n)$ by dominated convergence theorem. For the derivative we have that $$\partial_i (u(x)\theta(kx))  = \partial_i u(x)\theta(kx)+ku(x)\partial_i\theta(kx)$$
and the first term converges to $\partial_i u(x)$ by the same argument as above. For the second term notice that $\partial_i\theta(kx)$ converges to 0 a.e. and  
\[
\text{supp}(ku(x)\partial_i\theta(kx)) \subseteq  B\left(0,\frac{2}{k}\right) \cap B\left(0,\frac{1}{k}\right)^c.
\]
This implies that 
\[
|ku(x)\partial_i\theta(kx)| \leq \dfrac{2|u(x)\partial_i\theta(kx)|}{|x|},
\]
and thus $ku(x)\partial_i\theta(kx) \to 0$ a.e.. In addition, we have that
\[
|ku(x)\partial_i\theta(kx)| \leq \dfrac{2\|u\|_{\infty} \|\partial_i\theta\|_{\infty}}{|x|}\mathds{1}_{B(0,1)} \in L^2(\R^n),
\]
and we conclude by dominated convergence theorem.

When $u$ is radial we can take 
\[
\varphi(x) = \begin{cases}
\dfrac{1}{|x|^2-1} &\text{if }|x|<1 \\
0 & \text{if }|x| \geq 1, 
\end{cases}
\]
and define $\varphi_k(x) = Ck^n\varphi(kx)$, where $C>0$ is a normalization constant. It's well known that $(\varphi_n \star u) \to u$ in $H^1(\R^n)$, and the convolution of radial functions is also radial. Thus, is enough to apply the previous part of the lemma to the sequence $u_n = (\varphi_n \star u)$ to conclude.
\end{proof}

\subsection{Estimates in radial Sobolev spaces} The following result is standard in the literature, for completeness we include the proof.

\begin{lemma}\label{Sob_radial}
If $u \in H^1(\R^3)$ is radial; with abuse of notation, $u(x) =u(r)$. Then $u(r) \in L^2(0,\infty)$ and  $ru(r) \in L^p(0,\infty)$ for all $p \in [2,\infty]$. Moreover, we have the estimate
\[
\sup_{r \geq 0}|ru(r)| \leq C \|u\|_{H^1(\R^3)}.
\]
\end{lemma}

\begin{proof}
Since $u \in H^1$, from Hardy's inequality we have
\[
4\pi \int_0^\infty u^2(r)dr = \left\|\dfrac{u}{|x|}\right \|^2_{L^2(\R^3)} \leq 4\|\nabla u\|^2_{L^2(\R^3)} = 16\pi \int_0^\infty r^2u_r^2(r)dr.
\]
Hence $u(r) \in L^2(0,\infty)$ and one has the inequality
\[
 \int_0^\infty u^2(r)dr  \leq 4\int_0^\infty r^2u_r^2(r)dr.
\] 
In order to prove that $ru(r) \in L^\infty(0,\infty)$, notice first that $(ru(r))' = u(r)+ru_r(r)$ and since $\phi \in H^1(\R^3)$ from the previous part we have that $ru(r) \in H^1(0,\infty)$. By the Sobolev's embedding we conclude that $ru(r)$ is continuous and bounded. Using interpolation between $L^2$ and $L^\infty$, we conclude $ru(r) \in L^p(0,\infty)$ for all $p \in [2,\infty]$. Moreover, we have from Sobolev's embedding that
\begin{align*}
\sup_{r \geq 0}|ru(r)| \leq &~{} C\|ru(r)\|_{H^1(0,\infty)} \\
\leq &~{}  C(\|ru(r)\|_{L^2(0,\infty)}+\|u+ru_r\|_{L^2(0,\infty)}) \\
\leq &~{}  C(\|u\|_{L^2(\R^3)}+2\|u_r\|_{L^2(\R^3)}) \leq  2C\|u\|_{H^1(\R^3)},
\end{align*}
proving the desired estimate.
\end{proof}

\subsection{Estimates for the linear wave equation} Consider the linear equation associated with \eqref{eq:1} 

\begin{align}\label{eq: LW}
\partial_t^2\phi+3H\partial_t \phi-e^{-2Ht}\Delta \phi = f(t,x)
\end{align}
We can define some energy associated with this equation as \[
E(t) = \int \dfrac{\phi_t^2}{2}+\dfrac{|\nabla \phi|^2}{2e^{2Ht}}
\]To prove global existence for the nonlinear equation \eqref{eq:1} we shall need the following estimate.

\begin{lemma}\label{lema: LW}
Let $T>0$. If $f \in L^1([0,T],L^2(\R^3))$ then there exist $c \in \R$ such that every solution $(\phi, \phi_t) \in H^1 \times L^2$ satisfies 
\begin{align*}
E(t) \leq e^{ct}E(0)+\int_0^te^{c(t-s)}\|f(s)\|_{L^2}^2ds
\end{align*}
\end{lemma}

\begin{proof}
Notice that multiplying the equation \eqref{eq: LW} by $\phi_t$ we obtain that \[
\partial_t\left(\dfrac{\phi_t^2}{2}+\dfrac{|\nabla \phi|^2}{2e^{2Ht}}\right)-\text{div}\left(\dfrac{\phi_t\nabla \phi}{e^{2Ht}}\right)+3H\phi_t^2+H\dfrac{|\nabla \phi|^2}{e^{2Ht}} = f(t,x)\phi_t
\]Integrating on $\R^3$ we get 
\begin{align*}
\partial_tE(t) =& -H\int 3\phi_t^2+\dfrac{|\nabla \phi|^2}{e^{2Ht}}+\int f(t,x)\phi_t \\
\leq & -H\int 3\phi_t^2+\dfrac{|\nabla \phi|^2}{e^{2Ht}}+\int \dfrac{\phi_t^2}{2}+\int \dfrac{f(t,x)^2}{2} \\
\leq & cE(t)+\dfrac{\|f(t)\|_{L^2}^2}{2}
\end{align*}
Applying Gronwall's inequality we finish the proof. 
\end{proof}

\begin{remark}
Notice that when $H \geq \frac{1}{6}$ we see in the proof that the constant could be chosen $c \leq 0$, which would give us more control on the size of the solution.
\end{remark}
\subsection{Local and global existence} Finally, we recall the following results.

\medskip 

This first result deals with global solutions in the energy space in the case where the nonlinearity is Lipschitz. 

\begin{theorem}[\cite{Evans}]\label{thm:global}
If $f: \R \to \R$ is Lipschitz continuous and $f(0) = 0$ then for initial data in $H^1(\R^n) \times L^2(\R^n)$ equation (\ref{eq:1}) has a unique global solution such that
\begin{enumerate}
\item $\phi \in C([0,\infty); H^1(\R^n)) \cap C^1([0,\infty); L^2(\R^n)).$
\item If the initial datum $(\phi_0,\phi_1)\in H^1\times L^2$ is radial, then $(\phi,\partial_t \phi)$ is radial for all times.
\item Equation \eqref{Energy} is satisfied for all times $t \geq 0$.
\end{enumerate}
\end{theorem}

The following result follows closely Theorem \ref{thm:global}, but deals with globally defined small solutions in the case where the nonlinearity may have growth.  

\begin{theorem}\label{thm:local} The following holds:
\item If $f: \R \to \R$ is of class $C^2$ and $f(0) = f'(0) = 0$, then equation (\ref{eq:1}) has a unique maximal solution provided that the initial data $(\phi_0,\phi_1) \in H^2 \times H^1$ is small enough. This solution satisfies
\begin{enumerate}
\item $\phi \in C([0,T); H^2(\R^3)) \cap C^1([0,T); H^1(\R^1)).$
\item If the initial datum $(\phi_0,\phi_1)\in H^2\times H^1$ is radial, then $(\phi,\partial_t \phi)$ is radial for all times.
\item Equation \eqref{Energy} is satisfied along $0\leq t < T$.
\end{enumerate}
\end{theorem}

\begin{proof}
This result is standard, but for completeness we include a sketch of proof in Appendix \ref{A}.
\end{proof} 

\begin{theorem}[\cite{Sogge}]\label{thm: Global 2} 

Suppose that $f: \R \to \R$ is of class $C^2$, $f(0) = f'(0) = 0$ and there exists $C >0$ such that $$
|f''(s)| \leq C|s|^{p}, \quad |s| \leq 1
$$for $p \geq 1$. Then for smooth, compactly supported and small enough initial data, the equation \eqref{eq:1} for $H = 0$ has a unique global solution $\phi \in C^2(\R^{1+3})$.

\end{theorem}

\section{Virial identities}\label{viriales}

Now we prove some virial estimates needed for the proof of our main results. These are similar to the ones proved in Kowalczyk et al. \cite{KMM1}, and more precisely Alejo and Maul\'en \cite{AleMau}, but important differences appear in our case, where we use some particular weights that make virials completely defocusing, in a sense to be described below.

\subsection{First computations} 

For locally integrable functions $\psi(r),\varphi(t,r)$ to be chosen later let 
\begin{align}
\mathcal{P}(\phi )(t) &= \int_0^\infty \psi (r)\phi _r (t,r)\phi _t(t,r)dr,  \label{eq: virial_P} \\
\mathcal{R}(\phi )(t) &= \int_0^\infty \psi' (r) \phi(t,r) \phi _t(t,r)dr,  \label{eq: virial_R} \\
\mathcal I(\phi) (t)&= \mathcal{P}(\phi )(t)+\dfrac{1}{2}\mathcal{R}(\phi )(t), \label{eq: virial_I} \\ 
\mathcal{J}(\phi)(t) &= \int_0^\infty r^2\varphi(t,r)\left(\dfrac{\phi_t^2}{2}+\dfrac{\phi_r^2}{2e^{2Ht}}+F(\phi) \right). \label{eq: virial_J}
\end{align}

%
This functional has already been used in \cite{KMM1} and \cite{AleMau}. However, with respect to these previous references, we introduced new weights that are better adapted to the cosmological setting. In particular, our weights will also work for large data.

\medskip

The following is a standard but key computation:

\begin{lemma}\label{lem:virial1}
If $(\phi,\phi_t) \in H^1\times L^2$ is a radial solution of \eqref{eq:1} with $H = 0$ then 
\begin{align}
\dfrac{d\mathcal{P}(\phi)}{dt} =&~{} \int_0^\infty \dfrac{2\psi}{r}\phi_r^2- \int_0^\infty\psi'\left(\dfrac{\phi_t^2}{2}+\dfrac{\phi_r^2}{2}-F(\phi)\right), \label{dPdt}\\
\dfrac{d\mathcal{R}(\phi)}{dt} =&~{} \int_0^\infty \psi' \left(\phi_t^2-\phi_r^2  -\phi f(\phi) \right)   +\int_0^\infty \left(\dfrac{\psi'}{r^2}-\dfrac{\psi''}{r}+\dfrac{\psi'''}{2}\right) \phi^2 ,\label{dRdt}
\end{align}
and
\begin{equation}
\dfrac{d\mathcal{I}(\phi)}{dt} =\int_0^\infty \left( \left(\dfrac{\psi'}{r^2}-\dfrac{\psi''}{r}+\dfrac{\psi'''}{2}\right)\phi^2+\left(\dfrac{2\psi}{r}-\psi'\right)\phi_r^2+\dfrac{\psi'}{2}(2F(\phi)-\phi f(\phi)) \right). \label{dIdt}
\end{equation}
\end{lemma}

\begin{proof}[Proof of Lemma \ref{lem:virial1}] Thanks to Lemma \ref{lemma: density}, it is enough to compute all derivatives assuming data in $C_0^\infty(\mathbb R^3 \backslash \{0\})$.  

\medskip

Using equation \eqref{eq:1} and the definition of $F(s)=\int_0^s f(\sigma)d\sigma$, we have in \eqref{eq: virial_P}: 
\[ \begin{aligned}
\dfrac{d\mathcal{P}(\phi)}{dt} =& \int_0^\infty \psi \left(\phi_{rt}\phi_t+\phi_r\phi_{rr}+\dfrac{2}{r}\phi_r^2-\phi_rf(\phi) \right)dr \\
=& \int_0^\infty \psi \partial_r\left(\dfrac{\phi_t^2}{2}+\dfrac{\phi_r^2}{2}- F(\phi) \right)+\dfrac{2\psi}{r}\phi_r^2 \\
=& \psi \left(\dfrac{\phi_t^2}{2}+\dfrac{\phi_r^2}{2}-F(\phi) \right)\bigg\rvert_0^\infty+\int_0^\infty \dfrac{2\psi}{r}\phi_r^2-\psi'\left(\dfrac{\phi_t^2}{2}+\dfrac{\phi_r^2}{2}-F(\phi)\right).
\end{aligned} \]
Thanks to Lemma \ref{lemma: density}, every boundary term at zero and infinity disappear. We get \eqref{dPdt}.

\medskip

We compute now $\dfrac{d\mathcal{R}(\phi)}{dt} $.  We have from \eqref{eq: virial_R}:

\[ \begin{aligned}
\dfrac{d\mathcal{R}(\phi)}{dt} =& \int_0^\infty \psi'\left(\phi_t^2+\phi\phi_{rr}+\dfrac{2}{r}\phi\phi_r-\phi f(\phi)\right) \\
=& \int_0^\infty \psi'(\phi_t^2-\phi f(\phi))+\int_0^\infty \psi'\phi\phi_{rr}+\int_0^\infty\psi'\dfrac{2}{r}\phi\phi_r = : K_1+K_2+K_3.
\end{aligned} \]

$K_1$ is left as it is. We compute $K_2$ first saving every boundary term:
\[ \begin{aligned}
K_2 =& \psi'\phi\phi_r\bigg\rvert_0^\infty - \int_0^\infty \phi_r(\psi''\phi+\psi'\phi_r) \\
=& \psi'\phi\phi_r\bigg\rvert_0^\infty - \int_0^\infty \psi'\phi_r^2 -\int_0^\infty \psi''\partial_r\left(\dfrac{\phi^2}{2} \right) \\
=& \psi'\phi\phi_r\bigg\rvert_0^\infty - \int_0^\infty \psi'\phi_r^2 - \left(\dfrac{\psi''\phi^2}{2} \bigg\rvert_0^\infty - \int_0^\infty \psi'''\dfrac{\phi^2}{2} \right) \\
=& \psi'\phi\phi_r\bigg\rvert_0^\infty - \dfrac{\psi''\phi^2}{2} \bigg\rvert_0^\infty+\int_0^\infty \psi'''\dfrac{\phi^2}{2}-\int_0^\infty \psi'\phi_r^2.
\end{aligned} \]
Similarly,
\[ \begin{aligned}
K_3 =& \int_0^\infty \dfrac{\psi'}{r}\partial_r(\phi^2) = \dfrac{\psi'}{r}\phi^2\bigg\rvert_0^\infty - \int_0^\infty \partial_r\left(\dfrac{\psi'}{r}\right)\phi^2 \\
=&  \dfrac{\psi'}{r}\phi^2\bigg\rvert_0^\infty - \int_0^\infty \left(\dfrac{\psi''}{r}-\dfrac{\psi'}{r^2} \right)\phi^2.
\end{aligned} \]
Arranging all previous computations, we conclude that $\mathcal{R}(\phi)=\int_0^\infty \psi' \phi\phi_tdr$ satisfies:
\[
\begin{aligned}
&~{} \dfrac{d\mathcal{R}(\phi)}{dt} \\
&~{} = \int_0^\infty \psi'(\phi_t^2-\phi f(\phi)) \\
&~{} \quad + \psi'\phi\phi_r\bigg\rvert_0^\infty - \dfrac{\psi''\phi^2}{2} \bigg\rvert_0^\infty+\int_0^\infty \psi'''\dfrac{\phi^2}{2}-\int_0^\infty \psi'\phi_r^2 + \dfrac{\psi'}{r}\phi^2\bigg\rvert_0^\infty - \int_0^\infty \left(\dfrac{\psi''}{r}-\dfrac{\psi'}{r^2} \right)\phi^2.
\end{aligned}
\]
Again thanks to Lemma \ref{lemma: density}, every boundary term disappears. We obtain
\[
\begin{aligned}
\dfrac{d\mathcal{R}(\phi)}{dt}  &~{}= \int_0^\infty \psi'(\phi_t^2-\phi f(\phi))   +\int_0^\infty \psi'''\dfrac{\phi^2}{2}-\int_0^\infty \psi'\phi_r^2 - \int_0^\infty \left(\dfrac{\psi''}{r}-\dfrac{\psi'}{r^2} \right)\phi^2\\
&~{}= \int_0^\infty \psi' \left(\phi_t^2-\phi_r^2  -\phi f(\phi) \right)   +\int_0^\infty \left(\dfrac{\psi'}{r^2}-\dfrac{\psi''}{r}+\dfrac{\psi'''}{2}\right) \phi^2 .
\end{aligned}
\]
This proves \eqref{dRdt}. Finally, gathering \eqref{dPdt} and the previous identity in the definition of $\mathcal I(t)$ \eqref{eq: virial_I}, we arrive to \eqref{dIdt}.
\end{proof}

\begin{lemma}\label{lem: virial 2}
If $(\phi,\phi_t) \in H^1 \times L^2$ is a globally defined radial solution of \eqref{eq:1}, then 
\begin{align*}
\dfrac{d\mathcal{J}(\phi)}{dt} &= \int_0^\infty r^2 \varphi_t(t,r)\left(\dfrac{\phi_t^2}{2}+\dfrac{\phi_r^2}{2e^{2Ht}}+F(\phi)\right)\\
& -H\int_0^\infty r^2\varphi(t,r)\left(3\phi_t^2+\dfrac{\phi_r^2}{e^{2Ht}}\right)-\int_0^\infty r^2\varphi_r(t,r)\dfrac{\phi_t\phi_r}{e^{2Ht}}.
\end{align*}
\end{lemma}

\begin{proof}
As in Lemma \ref{lem:virial1} we shall assume data in $C_0^\infty(\R^3\setminus \{0\})$, and consequently, every boundary term disappear. Deriving \eqref{eq: virial_J} we have
\begin{align*}
\dfrac{d\mathcal{J}(\phi)}{dt} &= \int_0^\infty r^2\varphi_t(t,r)\left(\dfrac{\phi_t^2}{2}+\dfrac{\phi_r^2}{2e^{2Ht}}+F(\phi)\right) \\
&+ \int_0^\infty r^2\varphi(t,r)\left(\phi_{tt}\phi_t+\dfrac{\phi_r\phi_{rt}}{e^{2Ht}}-H\dfrac{\phi_r^2}{e^{2Ht}}+f(\phi)\phi_t\right) \\
&= K_1+K_2
\end{align*}
For the second term, using \eqref{eq:1} we get
\begin{align*}
K_2 &= \int_{0}^\infty r^2\varphi(t,r)\left(-3H\phi_t^2+\dfrac{\phi_t\Delta \phi}{e^{2Ht}}-f(\phi)\phi_t+\dfrac{\phi_r\phi_{rt}}{e^{2Ht}}-H\dfrac{\phi_r^2}{e^{2Ht}}+f(\phi)\phi_t\right)\\
&= -H\int_{0}^\infty r^2\varphi(t,r) \left(3\phi_t^2+\dfrac{\phi_r^2}{e^{2Ht}}\right)+\int_0^\infty r^2\varphi\left(\dfrac{\phi_r\phi_{rt}}{e^{2Ht}}+\dfrac{\phi_t\Delta \phi}{e^{2Ht}}\right) \\
&= -H\int_{0}^\infty r^2\varphi(t,r) \left(3\phi_t^2+\dfrac{\phi_r^2}{e^{2Ht}}\right)-\int_0^\infty r^2\varphi_r(t,r)\dfrac{\phi_r\phi_t}{e^{2Ht}},
\end{align*}
where in the last equality we have used the Green identity on $\R^3$. Arranging the previous calculations we got  
\begin{align*}
\dfrac{d\mathcal{J}(\phi)}{dt} = &~{} \int_0^\infty r^2\varphi_t(t,r)\left(\dfrac{\phi_t^2}{2}+\dfrac{\phi_r^2}{2e^{2Ht}}+F(\phi)\right) \\
&~{}  -H\int_0^\infty r^2\varphi(t,r)\left(3\phi_t^2+\dfrac{\phi_r^2}{e^{2Ht}}\right)-\int_0^\infty r^2\varphi_r(t,r)\dfrac{\phi_t\phi_r}{e^{2Ht}}.
\end{align*}
\end{proof}

\subsection{Choice of the weight function}

We will use Lemma \ref{lem:virial1} and \ref{lem: virial 2} with a particular choice of $\psi$ and $\varphi$. 
\begin{coro}\label{coro: 1}
Consider the weight
\begin{equation}\label{psi}
\psi(r) := \dfrac{r^2}{1 +r}.
\end{equation}
Then the following are satisfied:
\begin{enumerate}
\item  One has that
\begin{align} \label{Virial_bonito}
\mathcal{I}(\phi) = \int_0^\infty \dfrac{r^2}{1+r}\phi_r\phi_t+\dfrac{r(r+2)}{2(1+r)^2}\phi\phi_t 
\end{align}
is well-defined and bounded uniformly in time by the energy of the solution:
\[
\sup_{t\geq 0} \left| \mathcal{I}(\phi)(t) \right| \lesssim H[\phi ,\partial_t \phi ] (t=0).
\]
\item Also,
\begin{align}
\dfrac{d\mathcal{P}(\phi)}{dt} &= \int_0^\infty \dfrac{r(2+3r)}{2(1+r)^2}\phi_r^2-\dfrac{r(r+2)}{(1+r)^2}\left(\dfrac{\phi_t^2}{2} -F(\phi)\right) \label{dPdt bonito}\\ 
\dfrac{d\mathcal{R}(\phi)}{dt} &= \int_0^\infty \dfrac{r(r+2)}{(1+r)^2}(\phi_t^2-\phi f(\phi))+\dfrac{r(r+4)}{(1+r)^4}\phi^2-\dfrac{r(r+2)}{(1+r)^2}\phi_r^2 \label{dRdt bonito} \\
\dfrac{d\mathcal{I}(\phi)}{dt} &= \int_0^\infty r^2\left(\dfrac{1}{(1+r)^2}\phi_r^2+\dfrac{r+4}{2r(1+r)^4}\phi^2 \right)+\dfrac{r(r+2)}{2(1+r)^2}(2F(\phi)-\phi f(\phi)). \label{dIdt bonito}
\end{align}
\end{enumerate}
\end{coro}

\begin{proof}
The proof of \eqref{Virial_bonito}, \eqref{dPdt bonito}, \eqref{dRdt bonito} and \eqref{dIdt bonito} are direct from Lemma \ref{lem:virial1} and \eqref{psi}. We check now that $\mathcal{I}(\phi) $ is well-defined. Indeed,
\[
\left|  \int_0^\infty  \dfrac{r^2}{(1 +  r)}\phi_r\phi_t \right| \leq  \int_0^\infty r^2 \left| \phi_t \phi_r \right| \leq H(\phi,\phi_t).  
\] 
Additionally,
\[
\left|  \int_0^\infty \dfrac{r(r+2)}{2(1+r)^2}\phi\phi_t \right| \leq \int_0^\infty r |\phi\phi_t| \leq  H(\phi,\phi_t)^{1/2}\left( \int_0^\infty \phi^2 dr \right)^{1/2}.
\]
Finally, Lemma \ref{Sob_radial} gives the desired uniform in time proof. 
\end{proof}

\begin{coro}\label{coro 2}
For $\sigma, b \in \R$ consider the weight 
\begin{eqnarray}\label{varphi}
\varphi(t,r) = 1+\tanh(r+\sigma t+b).
\end{eqnarray}
Then we have the estimate
\begin{eqnarray}\label{eq: virial bonito 2}
\dfrac{d\mathcal{J}(\phi)}{dt} \leq (1+\sigma)\int_0^\infty r^2\sech^2(r+\sigma t + b) \left(\dfrac{\phi_t^2}{2}+\dfrac{\phi_r^2}{2e^{2Ht}}+F(\phi)\right) .
\end{eqnarray}
\end{coro}

\begin{proof}
Replacing \eqref{varphi} in Lemma \ref{lem: virial 2} we get 
\begin{align*}
\dfrac{d\mathcal{J}(\phi)}{dt} =&~{} \sigma \int_0^\infty r^2\sech^2(r+\sigma t + b)\left(\dfrac{\phi_t^2}{2}+\dfrac{\phi_r^2}{2e^{2Ht}}+F(\phi)\right) \\
& -H\int_0^\infty r^2(1+\tanh(r+\sigma t +b))\left(3\phi_t^2+\dfrac{\phi_r^2}{e^{2Ht}}\right) -\int_0^\infty r^2\sech^2(r+\sigma t + b)\dfrac{\phi_t\phi_r}{e^{2Ht}}.
\end{align*}
Noticing that the second term is strictly negative and 
\[
-\sech^2(r+\sigma t + b)\dfrac{\phi_r\phi_t}{e^{2Ht}} \leq \sech^2(r+\sigma t + b)\left(\dfrac{\phi_t^2}{2}+\dfrac{\phi_r^2}{2e^{2Ht}}\right),
\]
we have that 
\[
\begin{aligned}
\dfrac{d\mathcal{J}(\phi)}{dt} \leq &~{} \sigma \int_0^\infty r^2\sech^2(r+\sigma t + b)\left(\dfrac{\phi_t^2}{2}+\dfrac{\phi_r^2}{2e^{2Ht}}+F(\phi)\right) \\
&~{} +\int_0^\infty r^2\sech^2(r+\sigma t + b)\left(\dfrac{\phi_t^2}{2}+\dfrac{\phi_r^2}{2e^{2Ht}}\right) .
\end{aligned}
\]
Using that $F(\phi) \geq 0$ we obtain the desired estimate \eqref{eq: virial bonito 2}.
\end{proof}

If we define 
\[
\|\phi\|_{H^1_w}^2 = \int_0^\infty \dfrac{r^2}{(1 + r)^4}(\phi^2+\phi_r^2), \qquad \|\phi\|_{L^2_w}^2 = \int_0^\infty \dfrac{r^2}{(1 + r)^4}\phi^2
\]
we can see that
\[
\dfrac{d\mathcal{I}(\phi)}{dt} \geq \|\phi\|^2_{H^1_w}+\dfrac{r(r+2)}{2(1+r)^2}(2F(\phi)-\phi f(\phi))
\]
By other side, if we choose $\psi(r) = -\dfrac{3r^2+3r+1}{(1+r)^3}$ we have 
\[
\tilde{\mathcal{R}} = \int_0^\infty \dfrac{r^2}{(1+r)^4}\phi\phi_t,
\]
and 
\[
\dfrac{d\tilde{\mathcal{R}}}{dt}(\phi) = \int_0^\infty \dfrac{r^2}{(1+r)^4}(\phi_t^2-\phi_r^2-\phi f(\phi))+\dfrac{3r(3r-2)}{(1+r)^6}\phi^2.
\]
This allows us to prove the following propositions:

\begin{prop}\label{prop 1}
Under the hypothesis of Theorem \ref{thm: decaimiento} the solution $(\phi,\phi_t)$ of (\ref{eq:1}) satisfies 
\[
\int_0^\infty (\|\phi\|_{H^1_w}^2+\|\phi_t\|_{L^2_w}^2)dt < +\infty.
\]
\end{prop}
\begin{proof}
From Corollary \ref{coro: 1} and the previous calculations we see that 
\[
\dfrac{d\mathcal{I}(\phi)}{dt} \geq \|\phi\|_{H_w^1}^2,
\]
and then
\[
\begin{aligned}
\int_0^\infty \|\phi\|_{H^1_w}^2dt \leq &~{} \lim_{t \to \infty}\mathcal{I}(\phi(t))-\mathcal{I}(\phi(0))\\
 \lesssim &~{} H(\phi,\phi_t) + |\mathcal{I}(\phi(0))|. 
\end{aligned}
\]
On the other hand, since
\[
\dfrac{d\tilde{\mathcal{R}}}{dt}(\phi) = \int_0^\infty  \left( \dfrac{r^2}{(1+r)^4}(\phi_t^2-\phi_r^2-\phi f(\phi))+\dfrac{3r(3r-2)}{(1+r)^6}\phi^2 \right),
\]
we have 
\[ \begin{aligned}
\|\phi_t\|_{L^2_w}^2 =&~{} \dfrac{d\tilde{\mathcal{R}}}{dt}+\|\phi_r\|_{L^2_w}^2+\int_0^\infty \left(\dfrac{r^2}{(1+r)^4}\phi f(\phi)+ \dfrac{3r(2-3r)}{(1+r)^6}\phi^2 \right)\\
 \leq &~{} \dfrac{d\tilde{\mathcal{R}}}{dt}+\|\phi_r\|_{L^2_w}^2+\int_0^\infty \left(  \dfrac{Mr^2}{(1+r)^4} + \dfrac{3r(2-3r)}{(1+r)^6} \right) \phi^2 \\
 \leq &~{} \dfrac{d\tilde{\mathcal{R}}}{dt}+\|\phi_r\|_{L^2_w}^2+ C(M)\int_0^\infty  \dfrac{r(4+r)}{(1+r)^4} \phi^2 \\
 \leq & ~{}\dfrac{d\tilde{\mathcal{R}}}{dt}+\|\phi_r\|_{L^2_w}^2+C(M) \dfrac{d\mathcal{I}}{dt},
\end{aligned} \]
where $M$ is the Lipschitz constant of $f$, and we have used \eqref{dIdt bonito} and the previous observations. We note that 
\[
|\tilde{R}(\phi)| \leq \int_0^\infty \dfrac{r^2}{1+r}|\phi\phi_t| \leq \int_0^\infty r|\phi\phi_t|,
\]
and from the proof of Corollary \ref{coro: 1} we see that $\tilde{R}(\phi)$ is uniformly bounded in time by the energy of the solution. Integrating the last inequality the result follows. 
\end{proof}

\begin{prop}\label{prop 2}
Under the hypothesis of Theorem \ref{thm: decaimiento 2} the solution $(\phi,\phi_t)$ of (\ref{eq:1}) satisfies 
\[
\int_0^\infty (\|\phi\|_{H^1_w}^2+\|\phi_t\|_{L^2_w}^2)dt < +\infty.
\]
\end{prop}

\begin{proof}
From \eqref{crecimiento} we see that $F(\phi) \geq 0$ an using the inequality \eqref{crecimiento} we obtain that 
\[
\dfrac{d\mathcal{I}(\phi)}{dt} \geq \int_0^\infty r^2\left(\dfrac{1}{(1+r)^2}\phi_r^2+\dfrac{r+4}{2r(1+r)^4}\phi^2 \right)-C\dfrac{r(r+2)}{2(1+r)^2}\phi^4.
\]
Since we have supposed that $\sup_{t \geq 0}\|\phi(t)\|_{H^1\cap L^\infty} \leq \varepsilon$ we have 
\begin{align*}
|\phi(t,r)| \leq &~{} \varepsilon  \quad \forall t, r \geq 0, \\
\|\phi(t)\|_{H^1} \leq &~{} \varepsilon \quad  \forall t \geq 0,
\end{align*}
and from Lemma \ref{Sob_radial} we have that
\[
r|\phi(r)| \leq C\|\phi\|_{H^1} .
\]
Gathering both inequalities we obtain 
\[
\phi(r)^2 \leq \dfrac{(1+C)\varepsilon ^2}{(1+r)^2},
\]
and hence 
\[
\begin{aligned}
\dfrac{d\mathcal{I}(\phi)}{dt} \geq &~{} \int_0^\infty r^2\left(\dfrac{1}{(1+r)^2}\phi_r^2+\dfrac{r+4}{2r(1+r)^4}\phi^2 \right)-(1+C)\dfrac{r(r+2)}{2(1+r)^4}\varepsilon^2\phi^2 \\
 \geq &~{} \int_0^\infty \dfrac{r^2}{(1+r)^2}\phi_r^2+\dfrac{r(r+2)}{2(1+r)^4}(1-(1+C)\varepsilon^2)\phi^2 \\
 \geq &~{} \|\phi\|^2_{H_w^1},
\end{aligned}
\]
provided that $\varepsilon$ is small enough. Notice that, since $\phi(t,r)$ is uniformly bounded and $f$ is $C^1$ there existe a constant $M >0$ such that 
\[
|f(\phi(t))| \leq M|\phi(t)|, \quad \forall t \geq 0 ,
\]
which implies that $\phi f(\phi) \leq M\phi^2$, and we conclude as in the previous proposition.
\end{proof}

\section{Proof of main results}\label{proofs}

Now we can prove Theorem \ref{thm: decaimiento}, Theorem \ref{thm: decaimiento 2} and Theorem \ref{thm: decaimiento 3}.

\subsection{ Proof of Theorem \ref{thm: decaimiento} and \ref{thm: decaimiento 2}}
Let 
\[
\mathcal{H}(t) = \int_0^\infty \psi(\phi^2+\phi_r^2+\phi_t^2),
\]
then, we can see that 
\[ \begin{aligned}
\dfrac{d}{dt}\mathcal{H}(t) =& \int_0^\infty 2\psi(\phi\phi_t+\phi_r\phi_{rt}+\phi_t\phi_{tt}) \\
=&  \int_0^\infty2\psi\left(\phi\phi_t+\phi_r\phi_{rt}+\phi_t\phi_{rr}+\dfrac{2}{r}\phi_t\phi_r-\phi_tf(\phi)\right) \\
=& \int_0^\infty 2\psi\left(\phi\phi_t+\phi_r\phi_{rt}+\dfrac{2}{r}\phi_t\phi_r-\phi_tf(\phi)\right)dr +\int_0^\infty 2\psi \phi_t\phi_{rr}.
\end{aligned} \] 
Using Lemma \ref{lemma: density} we have 
\[
\int_0^\infty 2\psi \phi_t\phi_{rr} = -\int_0^\infty 2\phi_r(\psi \phi_{tr}+\psi'\phi_t),
\]
and hence
\[ 
\begin{aligned}
\dfrac{d}{dt} \mathcal{H}(t) =  2\int_0^\infty \psi(\phi\phi_t-\phi_tf(\phi))+\left(\dfrac{2\psi}{r}-\psi'\right)\phi_t\phi_r 
\end{aligned} 
\]
Since $\psi(r) = \dfrac{r^2}{(1+r)^4}$, we have that 
\[
\mathcal{H}(t) = \int_0^\infty\dfrac{r^2}{(1+r)^4}(\phi^2+\phi_r^2+\phi_t^2) = \|\phi\|_{H^1_w}^2+\|\phi_t\|_{L^2_w}^2,
\]
and 
\[
\dfrac{d}{dt}\mathcal{H}(t) = 2\int_0^\infty \dfrac{r^2}{(1+r)^4}(\phi\phi_t-\phi_tf(\phi))+\dfrac{4r^2}{(1+r)^5}\phi_r\phi_t.
\]
As we see, whether $f$ is globally Lipschitz or $f$ is $C^1$ and $\|\phi(t)\|_{L^\infty}$ is uniformly bounded in time then we have 
\[
|f(u)| \lesssim |u|.
\]
Thus, we see that 
\[ \begin{aligned}
\left|\dfrac{d}{dt}\mathcal{H}(t) \right|  \lesssim  & \int_0^\infty \dfrac{r^2}{(1+r)^4}(|\phi||\phi_t|+ |\phi_t||f(\phi)|)+\dfrac{4r^2}{(1+r)^5}\phi_r\phi_t \\
\lesssim & \int_0^\infty  \dfrac{r^2}{(1+r)^4}(\phi^2+\phi_t^2+ \phi^2) + \dfrac{4r^2}{(1+r)^4}(\phi_r^2+\phi_t^2) \\
\lesssim & ~{} \mathcal{H}(t).
\end{aligned} \]
From Proposition \ref{prop 1} and \ref{prop 2} there exists a sequence $t_n \to \infty$ such that $H(t_n) \to 0$. Integrating the inequality above on $[t,t_n]$we see that 
\[ \begin{aligned}
|\mathcal{H}(t_n)-\mathcal{H}(t)| =& \left|\int_t^{t_n}\dfrac{d}{dt}\mathcal{H}(s)ds\right| \\
\leq & \int_t^{t_n}\left|\dfrac{d}{dt}\mathcal{H}(s)ds\right| \lesssim  \int_t^{t_n}\mathcal{H}(s)ds ,
\end{aligned} \]
and passing to limit as $n \to \infty$ we have 
\[
\mathcal{H}(t) \leq \int_t^\infty \mathcal{H}(s)ds,
\]
and hence $\displaystyle\lim_{t \to \infty}\mathcal{H}(t) = \displaystyle\lim_{t \to \infty} \left( \|\phi\|_{H^1_w}^2+\|\phi_t\|_{L^2_w}^2 \right)= 0$. To conclude the proof is enough to note that for any $R>0$ we have  
\[ \begin{aligned}
\|(\phi,\phi_t)\|_{H^1 \times L^2(B(0,R))}^2 =& \int_{B(0,R)}(\phi^2+\phi_r^2+\phi_t^2)dx \\
=& ~{}4\pi \int_0^R r^2(\phi^2+\phi_r^2+\phi_t^2)dr \\
 \leq &~{} 4\pi(1+R)^4\int_0^R \dfrac{r^2}{(1+r)^4}(\phi^2+\phi_r^2+\phi_t^2)dr \\
\leq & ~{}4\pi(1+R)^4\mathcal{H}(t),
\end{aligned} \]
and the result follows.

\subsection{ Proof of Theorem \ref{thm: decaimiento 3}: Existence}
Let $\phi(t,x)$ be the local solution to \eqref{eq:1} on $[0,T)\times \R^3$ given by Theorem \ref{thm:local}. We shall prove that for $\varepsilon >0$ small enough we can extend this solution smoothly to $[0,T] \times \R^3$. First notice that we have the identity 
\[
\partial_t\left(\dfrac{\phi_t^2}{2}+\dfrac{|\nabla \phi|^2}{2e^{2e^{2Ht}}}+F(\phi)\right)-\text{div}\left(\dfrac{\phi_t\nabla \phi}{e^{2Ht}} \right) = -3H\phi_t^2-H\dfrac{|\nabla \phi|^2}{e^{2Ht}},
\]where the divergence is taken in the spatial variables. We denote 
\[
K(t_0,x_0) = \{(t,x) \in \R^4 \; | \; t\leq t_0, \quad H|x-x_0| \leq e^{-Ht}-e^{-Ht_0}\},
\] 
the backward light cone and 
\[
M(t_0,x_0) =  \{(t,x) \in \R^4 \; | \;  t\leq t_0, \quad H|x-x_0| = e^{-Ht}-e^{-Ht_0}\},
\]its lateral boundary. Then, integrating on \[
K_s^t(t_0,x_0) = \{(t,x) \in \R^4 \; | \; H|x-x_0| \leq e^{-Ht}-e^{-Ht_0}\} \cap [s,t]\times \R^3;
\] and applying the divergence theorem we see that 
\begin{eqnarray*}
\int_{B(x_0,R(t))}e(t)dx+\int_{M_s^t}\dfrac{1}{\sqrt{1+e^{2Ht}}}\left(\dfrac{\left|\phi_t\frac{x-x_0}{|x-x_0|}-\nabla \phi \right|^2}{2} +F(\phi)\right)dS = \\
\int_{B(x_0,R(s))}e(s)dx-\int_{M_s^t}3H\phi_t^2+H\dfrac{|\nabla \phi|^2}{e^{2Ht}} ,
\end{eqnarray*}
where $R(t) = \frac{e^{-Ht}-e^{-Ht_0}}{H}$, $M_s^t$ is the lateral boundary of the truncated cone $K_s^t$ and \[
e(t) = \left(\dfrac{\phi_t^2}{2}+\dfrac{|\nabla \phi|^2}{2e^{2e^{2Ht}}}+F(\phi)\right)(t,x)
\]is the energy density. The last identity implies that if $\phi = 0$ on $B(x_0,R(s))$ then $\phi = 0$ on $B(x_0,R(t))$ for every $t \in [s,t_0]$, that is, $\phi$ have finite speed of propagation and then, if the initial data have compact support then $\phi$ has it too for every time where it is defined. From the well posedness theory we have in addition that the solution is smooth in the spatial variables for every time.

Now we must to show that 
\[
\sup_{t \in [0,T)}\| \phi(t)\|_{\infty} <\infty.
\]
To do this we shall estimate the $H^2$ norm of $\phi$. Notice that from the hypothesis on the initial data there exist some constant such that 
\[
\|(\phi,\phi_t)(0)\|_{H^2 \times H^1} \leq \dfrac{C_0\varepsilon}{4}.
\]
To estimate $\|\phi(t)\|_{H^2}$ suppose that \[
\sup_{t \in [0,T)}\|\phi(t)\|_{H^2} \leq C_0\varepsilon 
\]for the same constant as above. We will show that we can improve this estimate to obtain that \[
\sup_{t \in [0,T)}\|\phi(t)\|_{H^2} \leq \dfrac{C_0\varepsilon}{2}.
\]For this we note that from \eqref{Energy} if $\varepsilon >0$ is small enough we have that \[
\dfrac{\|\phi_t\|^2_{L^2}}{2}+\dfrac{\|\nabla \phi\|_{L^2}^2}{2e^{2Ht}} \leq \varepsilon^2\left( \dfrac{\|h\|_{L^2}^2}{2}+\dfrac{\|\nabla g\|^2_{L^2}}{2e^{2Ht}}\right)+\varepsilon^4\int g^4,
\]and in consequence  \[
\sup_{t \in [0,T)}\|\phi(t)\|_{H^1} \leq (1+T)\dfrac{C_0(\varepsilon+\varepsilon^2)}{4}.
\]
To estimate the $H^2$ norm note that $u_i = \partial_{x_i}\phi$ satisfies the equation \[
\partial_t^2u_i+3H\partial_t u_i-\dfrac{\Delta u_i}{e^{2Ht}}+f'(\phi)u_i = 0,
\]
and from Lemma \ref{lema: LW} we have 
\begin{align*}
\dfrac{\|u_{it}\|_{L^2}^2}{2}+\dfrac{\|\nabla u_i\|^2}{2e^{2Ht}} \leq & \varepsilon^2 e^{ct}\left(\dfrac{\|\partial_{x_i}h\|_{L^2}^2}{2}+\dfrac{\|\partial_{x_i}\nabla g\|^2}{2e^{2Ht}}\right)+\int_0^t e^{c(t-s)}\|f'(\phi)u_i\|_{L^2}^2ds \\
\leq & \varepsilon^2 e^{ct}\left(\dfrac{\|\partial_{x_i}h\|_{L^2}^2}{2}+\dfrac{\|\partial_{x_i}\nabla g\|^2}{2e^{2Ht}}\right)+\varepsilon ^4\int_0^t e^{c(t-s)}\|\phi^2 u_i\|_{L^2}^2ds ,
\end{align*}
and in the same way as before we obtain that 
\[
\sup_{t \in [0,T)}\|\partial_{x_i}\phi(t)\|_{H^1} \leq C(T)\dfrac{C_0(\varepsilon+\varepsilon^2)}{4}.
\]
This implies, for $\varepsilon >0$ small enough, that 
\[
\sup_{t \in [0,T)}\|\phi(t)\|_{H^2} \leq \dfrac{C_0\varepsilon}{2},
\]
as we wanted. Note that applying the same argument as above, in addition to Lemma \ref{lema: LW} we obtain that all the spatial derivatives of $\phi$ are bounded in time. Additionally, since $u = \phi_t$ satisfies \[
\partial_t^2 u+3H\partial_t u-\dfrac{\Delta u}{e^{2Ht}}+2H\dfrac{\Delta \phi}{e^{2Ht}}+f'(\phi)u = 0
\] 
we can apply the same idea to prove that the $H^2$ norm of $\phi_t$ is uniformly bounded in time, due to since $\phi$ is uniformly bounded in $H^2$ then $\Delta \phi \in L^1([0,T]; L^2)$. This allow us to extend $\phi$ to a smooth function on $[0,T]$, and from the finite speed of propagation property we have that $\phi(T,\cdot)$ is a smooth function with compact support. In consequence we can extend the solution globally in time.

\subsection{ Proof of Theorem \ref{thm: decaimiento 3}: Decay}

We first prove \eqref{H00}. For $t_0>0$ $\sigma <-1$ and $b = 1+\sigma$ in Corollary \ref{coro 2}, which yields
\[
\mathcal{J}(\phi)(t) = \int_0^\infty r^2(1+\tanh(r+\sigma t+(1+\sigma)t_0))\left(\dfrac{\phi_t^2}{2}+\dfrac{\phi_r^2}{2e^{2Ht}}+F(\phi)\right),
\]
and
\[
\dfrac{d\mathcal{J}(\phi)}{dt}(t) \leq (1+\sigma)\int_0^\infty r^2\sech^2(r+\sigma t+(1+\sigma)t_0)\left(\dfrac{\phi_t^2}{2}+\dfrac{\phi_r^2}{2e^{2Ht}}+F(\phi)\right).
\]
Since $\sigma<-1$ we have that $\mathcal{J}(\phi)$ is decreasing on $[2,t_0]$. In particular 
\begin{align*}
& \int_0^\infty r^2(1+\tanh(r+(1+2\sigma)t_0))\left(\dfrac{\phi_t^2}{2}+\dfrac{\phi_r^2}{2e^{2Ht}}+F(\phi)\right)(t_0,r)\\
& \leq \int_0^\infty r^2(1+\tanh(r+2\sigma+(1+\sigma)t_0))\left(\dfrac{\phi_t^2}{2}+\dfrac{\phi_r^2}{2e^{2Ht}}+F(\phi)\right)(2,r) .
\end{align*}
Since $1+\sigma <0$ we have that 
\[
\lim_{t_0 \to \infty}(1+\tanh(r+2\sigma+(1+\sigma)t_0) )= 0.
\]
By dominated convergence theorem we have that the right hand side converges to 0 as $t_0 \to \infty$, and consequently we get 
\[
\lim_{t \to \infty}\int_0^\infty r^2(1+\tanh(r+(1+2\sigma)t))\left(\dfrac{\phi_t^2}{2}+\dfrac{\phi_r^2}{2e^{2Ht}}+F(\phi)\right)(t,r) = 0.
\]
To conclude it is enough to note that if $r >-(1+2\sigma)t$ then 
\[
1+\tanh(r+(1+2\sigma)t) \geq 1,
\]
and then 
\begin{align*}
\int_{R(t)} \left( \dfrac{\phi_t^2}{2}+\dfrac{\phi_r^2}{2e^{2Ht}}+F(\phi) \right)& \leq \int_{-(1+2\sigma)t}^\infty r^2(1+\tanh(r+(1+2\sigma)t))\left(\dfrac{\phi_t^2}{2}+\dfrac{\phi_r^2}{2e^{2Ht}}+F(\phi)\right) \\
& \leq \int_0^\infty r^2(1+\tanh(r+(1+2\sigma)t))\left(\dfrac{\phi_t^2}{2}+\dfrac{\phi_r^2}{2e^{2Ht}}+F(\phi)\right).
\end{align*}
Since we have proved that the right hand side converges to 0 we got the desired decay \eqref{H00}.

\medskip

Now we prove \eqref{H01}. Let $R>0$ fixed. Lemma \ref{lem: virial 2} with $\varphi(t,r)=\varphi(r) \geq 0$ yields for some $C>0$
\[
\begin{aligned}
\dfrac{d\mathcal{J}(\phi)}{dt} =&~{}  -H\int_0^\infty r^2\varphi(r)\left(3\phi_t^2+\dfrac{\phi_r^2}{e^{2Ht}}\right)-\int_0^\infty r^2\varphi'(r)\dfrac{\phi_t\phi_r}{e^{2Ht}}\\
\leq &~{} -3H\int_0^\infty r^2\phi_t^2 \left( \varphi(r) - C  e^{-Ht} \varphi'(r)\right) -H\int_0^\infty r^2\dfrac{\phi_r^2}{e^{2Ht}} \left(\varphi(r) -  \varphi'(r)\right) .
\end{aligned}
\] 
Choosing $\varphi_0(r), \varphi(r)$ satisfying
\[
\varphi_0(r)= 1, \quad r \leq 1, \quad \varphi_0'\leq 0, \quad \varphi_0(r) = e^{-r}, \quad r\geq 2, \quad \varphi(r)= \varphi_0(r/R), 
\]
If $R$ and $t$ are large,
\[
\dfrac{d\mathcal{J}(\phi)}{dt}  \leq -\frac32 H\int_0^\infty r^2\varphi(r)\phi_t^2 - \frac12 H\int_0^\infty r^2\varphi(r)\dfrac{\phi_r^2}{e^{2Ht}}  .
\]
Consequently, for some $t_0>0$ large,
\[
H\int_{t_0}^\infty \int_0^\infty r^2\varphi(r) \left(\phi_t^2+\dfrac{\phi_r^2}{e^{2Ht}} \right) dt <+\infty.
\]
Again, we have now
\[
\dfrac{d\mathcal{J}(\phi)}{dt} \lesssim  H\int_0^\infty r^2\varphi(r)\left(3\phi_t^2+\dfrac{\phi_r^2}{e^{2Ht}}\right),
\]
following the same ideas as in previous proofs, we get the desired result.

\section{Applications}\label{Applications}

In this section we shall see how Theorems \ref{thm: decaimiento} and \ref{thm: decaimiento 2} can be used to study the space-time dynamics of cosmological models. In particular, we shall prove here Corollaries \ref{coro1} and \ref{coro2}.

\subsection{Slow-roll in strongly defocusing models} First of all, we will need the following lemma about the potentials associated to certain models. Recall that from Section \ref{models} we have
\[ \begin{aligned}
F_{1,n}(s) =&~{} (1-e^{-s})^{2n},  \quad n =1,2,3,\ldots\\
F_{2,1}(s) =&~{} \tanh^2(s), \quad s\in\R, \\
F_{6,q}(s) =&~{} \dfrac{1}{q}[(1+s^2)^{q/2}-1 ], \quad   q \in [-1,1], \quad q\neq 0. 
\end{aligned} \]

\begin{lemma}\label{lemma: potenciales}
There exists a constant $C>0$ such that
\[ 
\begin{aligned}
s f_{1,n}(s) \leq & ~{}Cs^4, \quad  s \in [-1,1], \quad n \geq 2, \\
2F_{2,1}(s)-sf_{2,1}(s) \geq &~{} 0, \quad  s \in \R, \\
2F_{6,q}(s)-sf_{6,q}(s) \geq &~{} 0, \quad   q \in [-1,1], s \in \R.
\end{aligned} 
\]
\end{lemma} 
\begin{proof}
For $F_{2,1}$ by the symmetry of the function is enough to prove the inequality  for $s \geq 0$. We have 
\begin{align*}
2F_{2,1}(s) -sf_{2,1}(s) &= 2\tanh^2(s)-2s\tanh(s)\sech^2(s) \\
&= 2\tanh(s)(\tanh(s)-s\sech^2(s)),
\end{align*}
and since $|\tanh(s)| \leq |s|$ we have the inequality. For $F_{1,n}$ notice that for $s \in [-1,1]$
\[ 
\begin{aligned}
sf_{1,n}(s) =&~{} 2ns(1-e^{-s})^{2n-1}e^{-s} \\
\leq &~{} \left(\sup_{y \in [-1,1]} 2n(1-e^{-y})^{2n-4} \right) s (1-e^{-s})^3 ,
\end{aligned} 
\]
and therefore, from the classic inequality $1+s \leq e^s$ we see that for $s \in [0,1]$ we have 
\[
sf_{1,n}(s) \leq \left(\sup_{y \in [-1,1]} 2n(1-e^{-y})^{2n-4} \right) s^4.
\]
For $s \in [-1,0]$ we define 
\[
g_a(s) = s(1-e^{-s})^3 -as^4,
\]
and we prove that there exists some $a>0$ such that this $g_a(s) \leq 0$ on $[-1,0]$. Indeeed, a straightforward calculation gives us that 
\[
g_a(0) = g_a'(0) = g_a''(0) = g_a'''(0) = 0,
\]
and
\[
g_a^{(iv)}(s) = h(s)-24a,
\]
for some function $h(s)$. We can take $a >0$ as large as we need to guarantee that $g^{(iv)}(s) \leq 0$ on $[-1,0]$, and this implies that $g_a(s) \leq 0$ on $[-1,0]$, giving the desired inequality.

\medskip

For $F_{6,q}$ we have that 
\[
g(s) = 2F_{6,q}(s)-sf_{6,q}(s) = (1+s^2)^{q/2-1}\left[\dfrac{2}{q} +s^2\left(\dfrac{2}{q}-1\right)\right]-\dfrac{2}{q},
\]
and we notice that $g(0) = 0$. By the symmetry of the function it is sufficient to prove that $g'(s) \geq 0$ for $s \geq 0$
\[ 
\begin{aligned}
g'(s) =& f_{6,q}(s)-sf'_{6,q}(s) \\
=& \left(2-q\right)s^3(1+s^2)^{q/2-2},
\end{aligned} 
\]
and since $q \in [-1,1]$ we conclude.
\end{proof}

\begin{remark}
For the $T_n$-models and Axion Monodromy potential $F_{6,q}$ we always have global solutions for any initial data in $H^1\times L^2$ due to Theorem \ref{thm:global}.
\end{remark}

\begin{remark}
The potentials in Lemma \ref{lemma: potenciales} correspond to  the $E$-model, $T$-model and Axion-Monodromy models, respectively.  
\end{remark}

\begin{coro}[Nonexistence of standing solitons]
The scalar field models of finite energy solutions associated to $F_{1,1}$ and $F_{6,q}$, $q\in [-1,1]$, $q\neq 0,$ does not allow for nontrivial stationary solutions.
\end{coro}

\begin{proof}
This is a standard result. First note that from Pohozaev's indentity \cite{Willem}, every solution $\phi \in H^1(\R^n)$ to 
\[
-\Delta \phi +f(\phi) = 0,
\]
such that $F(\phi) \in L^1(\R^n)$ must satisfy
\[
\dfrac{n-2}{2}\int|\nabla \phi|^2+n\int F(\phi) = 0,
\]
and consequently it does not exist nontrivial stationary solutions of finite energy precisely when potential is positive. 
\end{proof}

\subsection{The general $T$ model}
For the $T$-model, when $n = 1$ we already know from Lemma \ref{lemma: potenciales} and Theorem \ref{thm: decaimiento} that every solution to \eqref{eq:1} with $H = 0$ satisfies \eqref{decaimiento}. For $n \geq 2$ the situation is different.

\begin{lemma}
We have $s f_{2,n}(s) \leq 2n s^4$, $s\in\R.$
\end{lemma}

\begin{proof}
We have
\begin{align*}
s f_{2,n}(s) &= 2ns\tanh^{2n-1}(u)\sech^2(s) \\
&= 2n s\tanh^3(s)\tanh^{2n-4}(s)\sech^2(s)  \leq 2n s^4.
\end{align*}
\end{proof}
Hence we can apply Theorem \ref{thm: decaimiento 2} to conclude that the same decay occurs as long as $\|\phi(t)\|_{H^1 \cap L^\infty}$ is small enough. This shows Corollary \ref{coro2} in the case of the $F_{2,n}$ model. 

\begin{remark}
Notice that from Theorem \ref{thm:global} if we set initial data in $H^2\times H^1$ we have a unique global solution to \eqref{eq:1} with $(\phi,\phi_t) \in H^2 \times H^1$ for every time.
\end{remark}

\begin{remark}
For the E-model we can apply again Theorem \ref{thm: decaimiento 2} to conclude the decay of small enough solutions to \eqref{eq:1} when $n \geq 2$. This is the missing part of Corollary \ref{coro2}. In this case it has not been possible to ensure that solutions are global in time for any initial data. However, Theorem \ref{thm: Global 2} assures the existence of global solutions when initial data is regular and small enough. 
\end{remark}

When $n = 1$ we cannot apply neither Theorem \ref{thm: decaimiento} nor \ref{thm: decaimiento 2} for $F_{1,1}$ because in this case we have no longer the required sign condition and $f_{1,1}'$ is not well behaved near the origin.


\subsection{The non-minimal coupling and Hilltop models}   For the non-minimal coupling model \[
F(\phi) = \lambda^2 \phi^4+\beta^2\phi^2
\]we give a slightly modified argument to the presented in Theorem \ref{thm: decaimiento 2} to conclude the decay of small enough solutions. Since $2F(\phi)-\phi f(\phi) = -\lambda^2 \phi^4$ we have in \eqref{dIdt bonito} that \[
\dfrac{d\mathcal{I}(\phi)}{dt} = \int_0^\infty r^2\left(\dfrac{1}{(1+r)^2}\phi_r^2+\dfrac{r+4}{2r(1+r)^4}\phi^2 \right)-\lambda^2\dfrac{r(r+2)}{2(1+r)^2}\phi^4 
\]
Supposing that $\sup_{t \geq 0}\|\phi(t)\|_{H^1\cap L^\infty} <\varepsilon$ and using that 
\[
\phi(r)^2 \leq \dfrac{(1+C)\varepsilon^2}{(1+r)^2},
\]
we have 
\[ 
\begin{aligned}
\dfrac{d\mathcal{I}(\phi)}{dt} =&~{} \int_0^\infty r^2\left(\dfrac{1}{(1+r)^2}\phi_r^2+\dfrac{r+4}{2r(1+r)^4}\phi^2 \right)-(1+C)\lambda^2\varepsilon^2\dfrac{r(r+2)}{2(1+r)^4}\phi^2 \\
 \geq &~{}  \int_0^\infty \left(\dfrac{r^2}{(1+r)^2}\phi_r^2+\dfrac{r(r+2)}{2(1+r)^4}\phi^2(1-(1+C)\lambda^2 \varepsilon^2) \right) \\
 \geq &~{} \|\phi\|_{H_w^1}^2,
\end{aligned}
\]
provided that $\varepsilon$ is small enough. Thus, we can conclude as in Theorem \ref{thm: decaimiento 2}. Notice that the hypothesis \eqref{acota} is also needed to prove that \[
\int_0^\infty \|\phi_t\|_{L^2_w}dt < \infty
\]as in Proposition \ref{prop 2}. For the Hilltop model with $n = 2$ we have a simpler situation. We have that $2F(\phi)-\phi f(\phi) = 2\phi^4$ and therefore 
\begin{align*}
\dfrac{d\mathcal{I}(\phi)}{dt} =& \int_0^\infty r^2\left(\dfrac{1}{(1+r)^2}\phi_r^2+\dfrac{r+4}{2r(1+r)^4}\phi^2 \right)+2\dfrac{r(r+2)}{2(1+r)^2}\phi^4 \\
\geq & \int_0^\infty r^2\left(\dfrac{1}{(1+r)^2}\phi_r^2+\dfrac{r+4}{2r(1+r)^4}\phi^2 \right) \\
\geq & \|\phi\|_{H^1_w}.
\end{align*}
Supposing \eqref{acota} we can conclude as in Theorem \ref{thm: decaimiento 2}.

\subsection{The Natural Inflation and Axion potentials} Assume small data $\phi$. Observe that in this case $F_{3,\pm} \geq 0$ in \eqref{F3}, and
\[
2F_{3,-} (\phi)- \phi f_{3,-}(\phi) = \frac1{12}\phi^4 + O(\phi^6).
\]
In the other case, we must take out the nonzero value of the potential at infinity to get finite energy, considering $F_{3,+} (\phi)= \cos\phi -1$. We get
\[
2F_{3,+} (\phi)- \phi f_{3,+}(\phi) = -\frac1{12}\phi^4 + O(\phi^6).
\]
In the former case, natural inflation, by virtue of \eqref{dIdt bonito} we get the desired result applying the same ideas as in the previous subsection. The case of Axion potential remains an interesting open problem.

\subsection{The D-brane model} Finally for the D-brane model $F_{4,n}$ in \eqref{F4}, since the potential and his derivative are singular in the origin we look for solutions of the form $\phi = 1+v$, where we suppose that $v(t) \in H^2$ and 
\[
\|v(t)\|_{H^2} < 1,
\]
for all times. Notice that this bound ensures control of the $L^\infty$ norm of $v$. In this case we have that the function $v$ satisfies (see \eqref{tF4})
\[
\partial_t^2v-\partial_r^2 v-\dfrac{2}{r}\partial_r v+2n \left( \dfrac{1}{(1+v)^{2n+1}} -1 \right)= 0.
\]
We conclude that 
\[
2 \tilde F_{4,n}(v) - v \tilde f_{4,n}(v)= \begin{cases} \frac{-2v^3(v+2)}{(1+v)^3}& n=1\\   -\frac{2 v^3 (10 + 15 v + 9 v^2 + 2 v^3)}{(1 + v)^5} & n=2. \end{cases}
\]
Consequently, Theorems \ref{thm: decaimiento} and \ref{thm: decaimiento 2} remain inconclusive in this setting. 

\subsection{The potential $F_7$} Now we consider the logarithm potential present in Corollary \ref{coro1}. Recall that 
\[
F_7(s) = \frac12\log(1+ s^2).
\]
Its derivative is clearly Lipschitz and from the well-known inequality (valid for $x>0$)
\[
1-\dfrac{1}{x} \leq \log(x),
\]
we have that 
\[
2F_7(\phi)-\phi f_7(\phi) = \log(1+\phi^2)-\dfrac{\phi^2}{1+\phi^2} \geq 0,
\]
and in consequence for any initial data in $H^1\times L^2$ the solution is global in time and satisfies \eqref{decaimiento}.

\appendix

\section{Proof of local and global existence}\label{A}

In this section, we sketch the proof of Theorem \ref{thm:global} and \ref{thm:local}. This proof is essentially taken from Evans \cite{Evans}.

\begin{proof}[Proof of Theorem \ref{thm:global}]
It is well known that for $H \geq 0$, $h \in L^1([0,T]; H^s)$ and $(f,g) \in H^{s+1}\times H^s$ the linear wave equation 
\[ \begin{aligned} 
\partial_t^2 \phi +3H\partial_t \phi -\dfrac{\Delta \phi}{e^{2Ht}}=&~{} h \\
(\phi(0), \phi_t(0)) =&~{} (f,g),  
\end{aligned} \]
has a unique solution $\phi \in C([0,T]; H^{s+1})\cap C^1([0,T], H^s)$ such that 
\begin{equation} \label{linear estimate}
\|\phi(t)\|_{H^{s+1}}+\|\phi_t(t)\|_{H^s} \leq C(1+T)\left(\|f\|_{H^{s+1}}+\|g\|_{H^s}+\int_0^t\|h(r)\|_{H^s}dr\right).
\end{equation}
See for example \cite{Sogge}. Using this, consider the Banach space 
\[
X = C([0,T]; H^1)\cap C^1([0,T], L^2),
\]
with norm 
\[
\|\phi\|_X = \sup_{t \in [0,T]}(\|\phi(t)\|_{H^1}+\|\phi_t\|_{L^2}),
\]
for some $T >0$ to be chosen later. We define the operator $A :X \to X$ by $Av = \phi$, where $\phi$ is the unique solution to 
\[ \begin{aligned}
\partial_t^2 \phi +3H\partial_t \phi -\dfrac{\Delta \phi}{a}=& f(v), \\
(\phi(0),\phi_t(0)) =& (h,g).
\end{aligned} \]
Note that $A$ is well defined because, since $f$ is globally Lipschitz and $f(0) = 0$, we have that $|f(v)| \leq M|v|$ and hence $f(v) \in L^1([0,T], L^2)$. Thus, given $v_1,v_2 \in X$ we have that $w = Av_1-A_v2$ satisfies 
\[ \begin{aligned}
\partial_t^2 w+3H\partial_t w-\dfrac{\Delta w}{a} =&~{} f(v_1)-f(v_2) \\
(w(0),w_t(0)) =&~{} (0,0)
\end{aligned} \] 
and so we can use the estimate \eqref{linear estimate}, obtaining 
\[ \begin{aligned}
\|w(t)\|_{H^1}+\|w_t(t)\|_{L^2} \leq&~{} C(1+T)\int_0^t\|f(v_1(s))-f(v_2(s))\|_{L^2}ds \\
\leq &~{} CM(1+T)t\sup_{s \in [0,T]}\|v_1(s)-v_2(s)\|_{L^2},
\end{aligned} \]
where $M >0$ is the Lipschitz constant of $f$. Thus we have that 
\[
\|w\|_X = \|Av_1-Av_2\|_X \leq CM(1+T)T\|v_1-v_2\|_X,
\]
and taking $T$ small enough such that $CM(1+T)T <1$ by Banach's fixed point theorem we have that there exist $\phi \in X$ such that 
\[ 
\begin{aligned}
\partial_t^2 \phi +3H\partial_t \phi-\dfrac{\Delta \phi}{e^{2Ht}} =&~{} f(\phi), \\
(\phi(0),\phi_t(0)) =&~{} (h,g).
\end{aligned} 
\]
Since $T$ does not depend on $h,g$ we can extend this solution for all times, concluding the proof.
\end{proof}
\begin{proof}[Proof of Theorem \ref{thm:local}]

Proceeding as above, consider the Banach space 
\[
X = C([0,T]; H^2) \cap C^1([0,T]; H^1),
\]
with norm 
\[
\|\phi\|_{X} = \sup_{t \in [0,T]}(\|\phi(t)\|_{H^2}+\|\phi_t(t)\|_{H^1}),
\]
and consider the subset $Y = \{\phi \in X \;| \; \|\phi\|_X \leq R\}$ with the metric induced by the norm of $X$. Define the operator $A: Y \to X$ by $Av = \phi$, where $\phi$ is the unique solution to 
\[ \begin{aligned}
\partial_t^2 \phi+3H\partial_t \phi -\dfrac{\Delta \phi}{e^{2Ht}} =&~{} f(v) \\
(\phi(0),\phi_t(0)) =&~{} (h,g),
\end{aligned} \]
where $(h,g) \in H^2\times H^1(\R^3)$. In order to see that $f(v) \in H^1$ notice that from Sobolev's embedding we have that $H^2(\R^3) \hookrightarrow C(\R^3) \cap L^\infty(\R^3)$, and then $|v(t,x)| \leq M$ for all $(t,x) \in [0,T]\times \R^3$. Note that this constant is independent of $v$, because $\|v\|_X \leq R $. Since $f$ is of class $C^2$ (and in particular locally Lipschitz continuous) we have 
\[
|f(v)| \lesssim |v|,
\]
and $f(v) \in L^2$. The same argument allows us to prove that
\[
|\nabla f(v)| = |f'(v)\nabla v| \lesssim |\nabla v|,
\]
and thus $f(v) \in L^1([0,T]; H^1)$ and $A$ is well defined. Using (\ref{linear estimate}) for $s = 1$ we see that
\[
\begin{aligned}
\|w\|_X \leq&~{} C(1+T)\left(\|h\|_{H^2}+\|g\|_{H^1}+\int_0^T\|f(v(s))\|_{H^1}ds \right) \\
\leq&~{} C(1+T)\left(\|h\|_{H^2}+\|g\|_{H^1}+TM\|v\|_X \right),
\end{aligned}
\]
where $M>0$ depends only on $f$ and $R$. Thus, taking $R \geq C(1+T)\|(h,g)\|_{H^2\times H^1}$ and $T>0$ small enough we have that $\|w\|_X \leq R$. Now, for $v_1,v_2 \in Y$ we have that $w = Av_1-Av_2$ satisfies
\begin{align*}
\partial_t^2 w-\Delta w =&~{} f(v_1)-f(v_2) \\
(w(0),w_t(0)) =&~{} (0,0),
\end{align*} 
and using (\ref{linear estimate}) again we have that 
\begin{align*}
\|w(t)\|_{H^2}+\|w_t(t)\|_{H^1} \leq&~{} C(1+t)\int_0^t\|f(v_1(s))-f(v_2(s))\|_{H^1}ds \\
\leq &~{} C(1+t)\int_0^t(\|f(v_1)-f(v_2)\|_{L^2}+\|f'(v_1)\nabla v_1-f'(v_2)\nabla v_2\|_{L^2})ds.
\end{align*}
Since both $f$ and $f'$ are locally Lipschitz continuous we have that 
\[
\|f(v_1)-f(v_2)\|_{L^2} \lesssim \|v_1-v_2\|_{L^2},
\]
and 
\begin{align*}
\|f'(v_1)\nabla v_1-f'(v_2)\nabla v_2\|_{L^2} \leq&~{} \|f'(v_1)\nabla v_1-f'(v_1)\nabla v_2\|_{L^2}+\|(f'(v_1)-f'(v_2))\nabla v_2\|_{L^2} \\
\lesssim &~{} \|\nabla v_1-\nabla v_2\|_{L^2}+\|v_1-v_2\|_{L^2}.
\end{align*}
This implies that
\[
\|w\|_{X} \leq CN(1+T)T\|v_1-v_2\|_{X}.
\]
and taking $T$ small enough we have that $A: Y \to Y$ is a contraction, and we conclude as in the proof of Theorem \ref{thm:global}. By standard arguments this solution can be extended to a maximal interval $[0,T)$.
\end{proof}







\begin{thebibliography}{99}

 
\bibitem{AleMau} M. A. Alejo, and C. Maul\'en, \emph{Decay for Skyrme wave maps}, Lett. Math. Phys. 112 (2022), no. 5, Paper No. 90, 33 pp.
 
\bibitem{Amin} M. A. Amin, R. Easther, H. Finkel, R. Flauger, and M. P. Hertzberg, \emph{Oscillons after Inflation}, Phys. Rev. Lett. 108, 241302 – Published 14 June 2012.
 





\bibitem{BrZh} E. Braaten, and H. Zhang, \emph{Axion stars}, preprint arXiv \url{https://arxiv.org/abs/1810.11473} (2018).

\bibitem{AX} F. Chadha-Day; J. Ellis; D. J. E. Marsh, \emph{Axion dark matter: What is it and why now?}. Science Advances. 8 (8): eabj3618 (23 February 2022), \url{https://arxiv.org/abs/2105.01406}.

\bibitem{Evans} L. C. Evans, \emph{Partial Differential Equations}, Second Edition. Graduate Studies in Mathematics, 19. American Mathematical Society, Providence, RI, 2010. xxii+749 pp. ISBN: 978-0-8218-4974-3.

\bibitem{H02} E. Di Valentino, O. Mena, S. Pan, L. Visinelli, W. Yang, A. Melchiorri, D. F. Mota, A. G. Riess, J. Silk, \emph{In the Realm of the Hubble tension--a Review of Solutions}, preprint 2021 \url{https://arxiv.org/pdf/2103.01183.pdf}.




\bibitem{Guth} A. H. Guth, {\it Inflationary universe: A possible solution to the horizon and flatness
problems}, Phys. Rev. D, vol. 23, pp. 347–356, (1981).

\bibitem{Hu} E. Hubble, \emph{A relation between distance and radial velocity among extra-galactic nebulae}, PNAS vol. 15 pp. 168--173, 1929.

%


\bibitem{Hong} Hong-Yi Zhang et al JCAP07 (2020) 055.

  
\bibitem{Kawa} M. Kawasaki, N. Kitajima, and T. T. Yanagida, \emph{Primordial black hole formation from
an axionlike curvaton model}, Phys. Rev. D, vol. 87, no. 6, p. 063519, 2013.


\bibitem{KMM1} M. Kowalczyk, Y. Martel, and C. Mu\~noz, {\it Nonexistence of small, odd breathers for a class of nonlinear wave equations}, Letters in Mathematical Physics, (2017) Vol. {107}, Issue 5, 921--931.

\bibitem{KMM3} M. Kowalczyk, Y. Martel, and C. Mu\~noz,  {\it On asymptotic stability of nonlinear waves}, Laurent Schwartz seminar notes (2017), see url at \url{http://slsedp.cedram.org/slsedp-bin/fitem?id=SLSEDP_2016-2017____A18_0}.


%

%

%
%
  
%






%
%
%
%
%



\bibitem{Lyth} D. H. Lyth, \emph{Axions and Inflation: Vaccuum Fluctuations}, Phys. Rev. D. Vol. 45 no. 10, May 15 1992, p. 3394, \url{https://doi.org/10.1103/PhysRevD.45.3394}.

\bibitem{Palma} G. Palma, S. Sypsas, and C. Zenteno, \emph{Seeding Primordial Black Holes in Multifield Inflation}, PRL 125, 121301 (2020).
  
\bibitem{PW} A. A. Penzias and R. W. Wilson, \emph{A Measurement of excess antenna temperature at
4080-Mc/s,} Astrophys. J., vol. 142, pp. 419--421, 1965.  
  
\bibitem{Planck2018} Planck Collaboration, \emph{Planck 2018 results. X. Constraints on inflation}, A\&A
Volume 641, September 2020.

\bibitem{H01} N. Sch\"oneberg, G. F. Abell\'an, A. P\'erez S\'anchez, S. J. Witte, V. Poulin, J. Lesgourgues, \emph{The $H0$ Olympics: A fair ranking of proposed models}, Physics Reports
Volume 984, 26 October 2022, Pages 1-55 \url{https://arxiv.org/abs/2107.10291}.
  
\bibitem{Sogge} C. D. Sogge, \emph{Lectures on Non-Linear Wave Equations}, Second edition. International Press, Boston, MA, 2008. x+205 pp. ISBN: 978-1-57146-173-5.

\bibitem{Starovinski}A. A. Starobinsky, {\it A New Type of Isotropic Cosmological Models Without Singularity}, Phys.
Lett. B 91, 99 (1980).

\bibitem{SW} P. Svrček, and E. Witten, \emph{Axions in string theory}. J. High Energy Phys. 2006, no. 6, 051, 52 pp. 

\bibitem{TW} K. Tsutaya, Y. Wakasugi, \emph{Blow up of solutions of semilinear wave equations related to nonlinear waves in de Sitter spacetime}, Partial Differential Equations and Applications 3, Article number: 6 (2022).
%
%


\bibitem{Willem} Mi. Willem, {\it Minimax Theorems}, Progress in Nonlinear Differential Equations and their Applications, 24. Birkhäuser Boston, Inc., Boston, MA, 1996. x+162 pp. ISBN: 0-8176-3913-6.

\bibitem{Zhang}H. Zhang, M. A. Amin, E. J. Copeland, P. M. Saffin and K. D. Lozanov, {\it Classical decay rates of oscillons} Journal of Cosmology and Astroparticle Physics (2020)





\end{thebibliography}
\end{document}